\documentclass[sn-mathphys-num]{sn-jnl}% Math and Physical Sciences Numbered Reference Style 
\usepackage{graphicx}%
\usepackage{multirow}%
\usepackage{amsmath,amssymb,amsfonts}%
\usepackage{amsthm}%
\usepackage{mathrsfs}%
\usepackage[title]{appendix}%
\usepackage{textcomp}%
\usepackage{manyfoot}%
\usepackage{booktabs}%
\usepackage{algorithm}%
\usepackage{algorithmicx}%
\usepackage{algpseudocode}%
\usepackage{listings}%
\usepackage{xcolor}
\usepackage{cancel}
\usepackage{mathtools}
\usepackage{subcaption}
\usepackage{enumitem}
\usepackage{makecell, tabularx}
\newtheorem{theorem}{Theorem}[section]
\newtheorem{corollary}[theorem]{Corollary}
\newtheorem{definition}[theorem]{Definition}
\newtheorem{lemma}[theorem]{Lemma}
\newtheorem{remark}[theorem]{Remark}

\newtheorem{example}[theorem]{Example}
\newtheorem*{claim*}{Claim}
\newtheorem{ancillary}[theorem]{Ancillary Info}

\begin{document}

\title[The Associated Discrete Laplacian and Mean Curvature]{The Associated Discrete Laplacian in $\mathbb{R}^3$ and Mean Curvature with Higher--order Approximations.}


\author*[]{\fnm{Wei-Hung} \sur{Liao}\orcid{https://orcid.org/0000-0002-9915-1030}}\email{roger2300245@gmail.com}
\affil*[]{\orgname{Shanghai Institute for Mathematics and Interdisciplinary Sciences}, 
%\orgaddress{\street{Block A, International Innovation Plaza, No. 657 Songhu Road}, \state{Yangpu District}, 
\city{Shanghai}, \postcode{200433}, \country{People's Republic of China}}


\abstract{\iffalse In the geometric processing of triangular meshes, the discrete Laplacian plays an important role in performing image data analysis, including mesh refinement, parameterization, segmentation, remeshing, simulation, etc. Among the many discretizations of the Laplacian, the cotangent Laplacian is widely used because it maintains most of the properties of the continuous Laplacian. It can be obtained from considering piecewise linear values on the original vertices or associated values on dual vertices, either of which results in the same discrete cotangent Laplacian in the triangular mesh setting.\fi In $\mathbb{R}^3$, the primal and dual constructions yield completely different discrete Laplacians for tetrahedral meshes.\iffalse We apply this geometrical property to propose a conformal algorithm to obtain a conformal mapping between a closed genus--zero surface and a sphere.\fi In this article, we prove that the discrete Laplacian satisfies the Euler-Lagrange equation of the Dirichlet energy in terms of the associated discrete Laplacian corresponding to the dual construction. Specifically, for a three--simplex immersed in $\mathbb{R}^3$, the associated discrete Laplacian on the tetrahedron can be expressed as the discrete Laplacian of the faces of the tetrahedron and the associated discrete mean curvature term given by the ambient space $\mathbb{R}^3$. Based on geometric foundations, we provide a mathematical proof showing that the dual construction gives a optimal Laplacian in $\mathbb{R}^3$ compared to the primal construction. Moreover, we show that the associated discrete mean curvature is more sensitive to the initial mesh than other state-of-the-art discrete mean curvatures when the angle changes instantaneously. Instead of improving the angular transient accuracy through mesh subdivision, we can improve the accuracy by providing a higher--order approximation of the instantaneous change in angle to reduce the solution error.}




\maketitle

\section{Introduction}%DONE
With the rapid development of imaging techniques and the proliferation of sensor-acquired 3D images, techniques for processing and analysing image data have become very important. In recent years, the discretization of differential operators has been an exciting area of numerical analysis, geometric processing and computational engineering. The energy function space of a manifold is determined by the geometrical characteristics of the manifold. By studying its Euler-Lagrange equations, it can reflect much of the global geometric behaviour of the manifold. One of the energy functions commonly considered is the Dirichlet energy, so the discrete Laplacian becomes very important in numerical analysis, computer graphics and geometry processing. According to the discretization of Laplacian, discrete operators will inherit different properties. For triangular surface meshes, Wardetzky et al. \cite{wardetzky2007} classified several important properties and analysed the connection between these properties for different discretization and continuous operators. The cotangent discretization is widely used in geometric processing, and it satisfies all of the above desirable properties if the mesh has Delaunay properties, and for any fixed set of function values it also minimizes the Dirichlet energy of the piecewise linear interpolation function, lacking the maximum principle only if the mesh is not Delaunay. There are several ways to derive the cotangent Laplacian on simplicial meshes.

For triangle meshes, the primal and dual constructions of the Laplacian lead to the same discrete cotangent--weight Laplacian, as shown in \cite{pinkall1993}. The familiar formula for the primal construction of the Laplacian in dimensions equal to $3$ first appeared in \cite{desbrun1999}, but they omitted the important length term. Later, Gu et al. \cite{gu2003volumetric} gave a definition of the primal discrete Laplacian on a tetrahedral mesh for the study of heat flow problems to minimise harmonic energy, and then Chao et al. \cite{chao2010simple} considered the same discrete Laplacian in the context of elasticity and function interpolation. For higher dimensions, the primal discrete Laplacian is considered in \cite{crane2019n}.

In 2020, Alexa et al. \cite{alexa2020properties} provided explicit formula for the lesser-known dual construction of the Laplacian on tetrahedral meshes. The primal and dual approaches produce the same discrete Laplacian in triangular meshes. However, they have completely different representations and properties in tetrahedral meshes.

\begin{definition}
Suppose $\mathcal{N}^3$ is a tetrahedral mesh embedded in $\mathbb{R}^3$. A discrete Laplacian defined on $\mathcal{N}^3$ is said to be optimal if the sum of the discrete Laplacian and discrete mean curvature vector defined on the two-dimensional manifold $\partial\mathcal{N}^3$ converges to the discrete Laplacian on $\mathcal{N}^3$ in a weak* sense.\footnote{Let $\{f_n\}$ be a sequence of bounded linear functionals on a norm space $X$ and there exists an $f\in X^*$ such that $f_n(x)\longrightarrow f(x)$ for all $x\in X$; namely, $f_n\stackrel{\ast}{\rightharpoonup} f$.}
\end{definition}

\iffalse
\textcolor{red}{
\begin{definition}\cite[Def. 4.9-4]{kreyszig1991}
Let $\{f_n\}$ be a sequence of bounded linear functionals on a norm space $X$. Then weak* convergence of $\{f_n\}$ means that there is an $f\in X^*$ such that $f_n(x)\longrightarrow f(x)$ for all $x\in X$. This can be expressed by 
\begin{align*}
f_n\stackrel{\ast}{\rightharpoonup} f.
\end{align*} 
\end{definition}}
\fi

\iffalse
\textcolor{blue}{
\begin{remark}
For general simplicial meshes there exists no discrete Laplace operator that satisfies all of the desired properties "Locality", "Symmetry", "Linear Precision", and "Positive Weights" simultaneously. This limitation provides a categorization of the existing literature and explains the large number of existing discrete Laplacians. Since not all desired properties can be satisfied at the same time, the design of a discrete Laplacian depends on the application at hand in order to satisfy the particular needs of a specific problem. In the sense of properly following the finite element paradigm, the primal Laplacian has its own merits.
\end{remark}}
\fi

The main contributions of this paper are in the following properties:
\begin{itemize}
    \item Based on the geometric fact of hypersurfaces in $\mathbb{R}^3$\iffalse discussed in \hyperref[AppendixA]{Appendix A}\fi, the discrete Laplacian on a triangular mesh can define the discrete mean curvature at vertices or on edges (Section~\ref{sec:4}). With respect to the discrete Laplacian on a tetrahedral mesh, as an approach for the continuous Laplacian in $\mathbb{R}^3$, it must contain a term representing the mean curvature of the hypersurface. On the basis of the dual construction of discrete Laplacian, we define its associated discrete Laplacian and reciprocal discrete mean curvature and endow it with some notion of the geometry of surface curvature(Section~\ref{sec:5}).
    \item We give a mathematical proof that the associated discrete Laplacian satisfies the Euler-Lagrange equation for the Dirichlet energy. That is, for a three--simplex immersed in $\mathbb{R}^3$, the associated discrete Laplacian on the tetrahedron can be expressed as the discrete Laplacian of the faces of the tetrahedron and the discrete mean curvature term given by the ambient space (Section~\ref{sec:6}).
    \item For geometric PDEs, reducing the mesh size allows the mesh to better reflect the surface structure, thus minimizing the discretization error. Ending the computation early before solving the PDE accurately can lead to solution errors. In addition to improving the mesh quality, it is important to ensure that the approximate order of the geometries is not less than the approximate order of the mesh. In Section~\ref{sec:7}, for the mean curvature equation, we provide higher--order approximations of the mean curvature $|\vec{H}|$ to minimize the solution error.
\end{itemize}

%-------------------------------------------------------------------------
\section{Preliminary and overview}%DONE
Let $\left(\mathcal{M}, \gamma\right)$ and $\left(\mathcal{N}, g\right)$ be two connected Riemannian manifolds of dimension $n$ and $n+d$  with metric tensors $\left(\gamma_{\alpha\beta}\right)$ and $\left(g_{ij}\right)$ in the local charts $x=\left(x_1, \dots, x_n\right)$ and $f=\left(f_1, \dots, f_{n+d}\right)$ on $\mathcal{M}$ and $\mathcal{N}$, respectively.

%-------------------------------------------------------------------------
\subsection{Euler-Lagrange equation for the Dirichlet energy}
Let $f$ be a map in $C^{\infty}(\mathcal{M}, \mathcal{N})$. In practice, the manifold $\mathcal{N}$ can be chosen as an isometric embedding of $\mathbb{R}^{n+d}$ with $g\equiv\delta$ being the Euclidean metric, then the Dirichlet (String) energy of the map $ f=\left(f^1, \cdots, f^n, \cdots, f^{n+d}\right)$ is defined as the integral of the square of $df$ as follows,
\begin{align}\label{Diri}
E_D(f)
\equiv&\frac{1}{2}\int_{\mathcal{M}}|d f|^2 d\mathrm{vol}_{\mathcal{M}}=\frac{1}{2}\int_{\mathcal{M}}\gamma^{\alpha\beta}(x)\delta_{ij}(f)\frac{\partial f^i}{\partial x^{\alpha}}\frac{\partial f^j}{\partial x^{\beta}}\sqrt{\gamma}dx.
\end{align}
where $\gamma=\det\left(\gamma_{\alpha\beta}\right)$ and $d\mathrm{vol}_{\mathcal{M}}=\sqrt{\gamma}dx$ is the volume element on $\mathcal{M}$.

Given any point $x\in\mathcal{M}$, we can find an open ball  $\mathbb{B}_r(x)\subset\mathcal{M}$ such that $\mathbb{B}_r(x)$ is mapped into a coordinate chart $U$ of $\mathcal{N}$ by a smooth map $f$. In these coordinate systems, $f$ can be viewed as a map $f:\mathbb{B}_r(x)\to U\subset\mathbb{R}^n$ and the maps $f+t\phi$ for $\phi\in C^{\infty}_{o}\left(\mathbb{B}_r(x), U\right)$ and $|t|$ sufficiently small determine the variation of the Dirichlet energy in \eqref{Diri}. By the calculus of variations \cite[Chap 8]{jost2011}, we know that the Dirichlet energy in \eqref{Diri} achieves a critical value if and only if a critical point $f$ satisfies the following Euler--Lagrange equation componentwisely.
\begin{align}
\label{eq:EL}
\Delta_{\mathcal{M}} f^m+g^{ij}\Gamma_{st}^m(f(x))\frac{\partial f^s}{\partial x^{i}}\frac{\partial f^t}{\partial x^{j}}=0,\quad m=1,\cdots n+d.
\end{align}
where 
\begin{align*}
\Delta_{\mathcal{M}}f^m:=\frac{1}{\sqrt{\gamma}}\frac{\partial}{\partial x^{\alpha}}\left(\sqrt{\gamma}\gamma^{\alpha\beta}\frac{\partial}{\partial x^{\beta}}f^m\right),\quad\Gamma_{st}^m:=\frac{1}{2}g^{ml}\left(g_{sl, t}+g_{tl, s}-g_{st, l}\right)
\end{align*}
is the Laplace-Beltrami operator on $\mathcal{M}$ and the Christoffel symbos in $\mathcal{N}$, respectively. Moreover, $f\in C^2\left(\mathcal{M}, \mathcal{N}\right)$ are called harmonic map, if it is the solutions of nonlinear partial differenial equations \eqref{Diri} with the Laplace-Beltrami operator as the principal term and the gradient of the solution as a nonlinear quadratic term.

Another idea is to apply the projection $\pi:\mathbb{R}^{n+d}\to\mathcal{N}$ to derive the equivalent form of \eqref{eq:EL}, which maps the points in $\mathbb{R}^{n+d}$ to the nearest points in $\mathcal{N}$.  Similarily, we practically assume $\mathcal{N}$ to be isometrically embedded in $\mathbb{R}^{n+d}$ and consider the local variation of the Dirichlet energy around the point $x\in\mathcal{M}$ with respect to the map $f:\mathcal{M}\to\mathcal{N}$. For any $x\in\mathbb{B}_r(x)\subset\mathcal{M}$ and $\phi\in C^{\infty}_o\left(\mathcal{M},\mathbb{R}^{n+d}\right)$, we can split the variational vector field $\phi$ into the tangential and normal parts with respect to $T_{y}\mathcal{N}$, $y=f(x)$ as follows.\\
Let $\left\{\mu_{n+1}, \dots, \mu_{n+d}\right\}$ be a basis spanning $\left(T_{y}\mathcal{N}\right)^{\perp}$,
\begin{equation*}
\begin{split}
\phi^{\perp}(x):=&\sum_{m=1}^d\phi^{n+m}\mu_{n+m}\in  \left(T_{y}\mathcal{N}\right)^{\perp},\ \phi^{n+m}(x):=\phi(x)\cdot\mu_{n+m}\\
\phi^{\top}(x):=&\phi(x)-\phi^{\perp}(x)\in T_{y}\mathcal{N}.
\end{split}
\end{equation*}
The variation of $E_D(f)$ at point x is given by,
\begin{align}
\delta E_D(f)
:=&\frac{d}{dt}\left(\frac{1}{2}\int_{\mathcal{M}}\gamma^{\alpha\beta}\nabla_{\frac{\partial}{\partial x^{\alpha}}}\pi\left(f+t\phi\right)\nabla_{\frac{\partial}{\partial x^{\beta}}}\pi\left(f+t\phi\right)\sqrt{\gamma}dx\right)\bigg\rvert_{t=0}\nonumber\\
%=&\textcolor{red}{\frac{d}{dt}\frac{1}{2}\int_{\mathcal{M}}\gamma^{\alpha\beta}(x)\frac{\partial\pi\left(f+t\phi\right)}{\partial x^{\alpha}}\frac{\partial\pi\left(f+t\phi\right)}{\partial x^{\beta}}\sqrt{\gamma}dx\rvert_{t=0}}\nonumber\\
%=&\textcolor{red}{\left(\int_{\mathcal{M}}\gamma^{\alpha\beta}(x)\frac{\partial\pi\left(f\right)}{\partial x^{\alpha}} d\pi\left(f(x)\right)\frac{\partial \phi^{j}}{\partial x^{\beta}}\sqrt{\gamma}dx\right)e_{ij}}\nonumber\\
%=&\int_{\mathcal{M}}\gamma^{\alpha\beta}\left[d\pi(y)\nabla_{\beta}\phi\right]\left[\nabla_{\alpha}f(x)\right]\sqrt{\gamma}dx\nonumber\\
%=&\textcolor{red}{\int_{\mathcal{M}}\gamma^{\alpha\beta}\nabla_{\alpha}f(x)\left[d\pi(y)\left(\nabla_{\beta}\phi\right)^{\top}\right]\sqrt{\gamma}dx}\nonumber\\
%  &\textcolor{red}{+\int_{\mathcal{M}}\gamma^{\alpha\beta}\nabla_{\alpha}f(x)\cancelto{0}{\left[d\pi(y)\left(\nabla_{\beta}\phi\right)^{\perp}\right]}\sqrt{\gamma}dx}\nonumber\\
=&\int_{\mathcal{M}}\gamma^{\alpha\beta}\left[\nabla_{\alpha}f\right]\left[\nabla_{\beta}\phi\right]\sqrt{\gamma}dx\label{2ndform}\\
%=&\textcolor{red}{\int_{\mathcal{M}}\gamma^{\alpha\beta}\nabla_{\alpha}f(x)\nabla_{\beta}\left(\phi^{\top}+\phi^{\perp}\right)\sqrt{\gamma}dx}\nonumber\\
%=&\textcolor{red}{\int_{\mathcal{M}}-\Delta_{\mathcal{M}}f(x)\left[\phi^{\top}+\phi^{\perp}\right]\sqrt{\gamma}dx}\nonumber\\
%=&\textcolor{red}{\int_{\mathcal{M}}-\Delta^{\top}_{\mathcal{M}}f(x)\phi\sqrt{\gamma}dx+\int_{\mathcal{M}}-\Delta^{\perp}_{\mathcal{M}}f(x)\phi\sqrt{\gamma}dx}\nonumber\\
=&\int_{\mathcal{M}}\left(-\Delta^{\top}_{\mathcal{M}}f\right)\phi\sqrt{\gamma}dx+\sum_{m=1}^d\int_{\mathcal{M}}\gamma^{\alpha\beta}l_{\mu_{n+m}}\left(\nabla_{\alpha}f, \nabla_{\beta}f\right)\phi^{n+m}\sqrt{\gamma}dx\label{second1}\\
=&\int_{\mathcal{M}}\left(-\Delta^{\top}_{\mathcal{M}}f\right)\phi\sqrt{\gamma}dx+\int_{\mathcal{M}}\gamma^{\alpha\beta}l\left(\nabla_{\alpha}f, \nabla_{\beta}f\right)\phi\sqrt{\gamma}dx\label{second2}
\end{align}
By the definition of projection $\pi$, we have
$
\pi\circ f(x)=f(x)$, $d\pi(y)X=X
$ 
for all $x\in\mathcal{M}$ and $X\in T_y\mathcal{N}$ with $y=f(x)$. Thus, the equality \eqref{2ndform} is first obtained with $d\pi(y)\nabla_{\beta}\phi=\left[\nabla_{\beta}\phi\right]^{\top}$ and then the tangent symbol can be dropped because $\nabla_{\alpha}f(x)\in T_y\mathcal{N}$ i.e., $\nabla_{\alpha}f(x)\left(\nabla_{\beta}\phi\right)^{\top}=\nabla_{\alpha}f(x)\nabla_{\beta}\phi$. Next, the second integral term of \eqref{second1} is derived from the concept of the second fundamental form\footnote{For $X, Y\in T_x\mathcal{M}$ and $\mu\in T_x\mathcal{M}^{\perp}$, the second fundamental form is the map $l: T_x\mathcal{M}\times T_x\mathcal{M}\rightarrow T_x\mathcal{M}^{\perp}$
\begin{align*}
l\left(X, Y\right)=l_{\mu}\left(X, Y\right)\mu,\quad  l_{\mu}\left(X, Y\right):=\left\langle(\nabla_{X}\mu)^{\top}, Y\right\rangle,
\end{align*}
where $S_{\mu}(X):=(\nabla_{X}\mu)^{\top}$ is also called the shape operator}, which describes the shape of $\mathcal{M}$ embedding into $\mathcal{N}$. 

Based on the variational argument in \eqref{second2}, the Dirichlet energy in \eqref{Diri} achieves a critical value if and only if a critical point $f$ satisfies the Euler--Lagrange equation componentwisely
\begin{align}\label{DiriEL}
\Delta_{\mathcal{M}} f^m-\gamma^{\alpha\beta}l^m\left(\nabla_{\alpha} f,\nabla_{\beta} f \right)=0,\quad m=1, \cdots, n+d.
\end{align}

\begin{remark}
From the above discussion of the Euler-Lagrange equations, the first idea to derive the Euler-Lagrange equation \eqref{eq:EL} is based on the local coordinate systems $\mathbb{B}_r(x)$ of $\mathcal{M}$ and $f(\mathbb{B}_r)\subset U$ of $\mathcal{N}$. The quadratic nonlinearity in the gradient, i.e., $g^{ij}\Gamma_{st}^m(f(x))\frac{\partial f^s}{\partial x^{i}}\frac{\partial f^t}{\partial x^{j}} $ in \eqref{eq:EL}, reflects the nonlinear local geometry between $\mathcal{M}$ and its image $\mathcal{N}$. An intuitive way to study local geometry is to control the Christoffel symbols $\Gamma^m_{st}$, but they only describe how the basis varies in a given coordinate system. Besides, the Christoffel symbols cannot be controlled by some geometric quantities until the solutions $f$ of \eqref{eq:EL} are determined. In other words, $g^{ij}\Gamma_{st}^m(f(x))\frac{\partial f^s}{\partial x^{i}}\frac{\partial f^t}{\partial x^{j}}$  after we determine the solution $f$ and construct the local coordinate in the image $\mathcal{N}$ which makes the study of solutions to \eqref{eq:EL} more difficult. 
The second idea in deriving \eqref{DiriEL} is based on the projections $\pi$ from $\mathbb{R}^{n+d}$ to the tangent space of $\mathcal{N}:=f(\mathcal{M})$, where the coordinate system is vague. According to the derivation of \eqref{DiriEL}, we can see the nonlinearity, i.e., $\gamma^{\alpha\beta}l^m\left(\nabla_{\alpha} f,\nabla_{\beta} f \right)$, which is the global geometry of $\mathcal{N}$ induced by the ambient space $\mathbb{R}^{n+d}$.

In applications, the domain and image are often treated as the known conditions and we seek the mapping is to satisfy some geometric properties  from the domain to image. Therefore it seems reasonable and feasible to use the projective viewpoint to determine the Euler-Lagrange equations.
\end{remark}

%----------------------------------------------------------------------
\subsection{Discrete approximation to the Dirichlet energy}
Consider a map $f\in C^{\infty}_o\left(\mathcal{M}, \mathbb{R}^{n+d}\right)$, which vanishes on the boundary $\partial \mathcal{M}$, the Dirichlet energy in \eqref{Diri} can be re-expressed in terms of
\begin{align*}
\frac{1}{2}\int_{\mathcal{M}}|d f|^2 d\mathrm{vol}_{\mathcal{M}}=\left(-\Delta f, f\right)_{L^2},
\end{align*}
where $(\cdot)_{L^2}$ is the inner product of the induced $L^2$-norm. 
Based on the notion of an approximation to the Dirichlet energy
\begin{align}\label{approach}
\left(-\Delta f, f\right)_{L^2} \approx -\mathbf{f}^{\intercal}\mathbf{L}\mathbf{f},\quad\mbox{$\mathbf{f}^{\intercal}$ is the transpose of $\mathbf{f}$},
\end{align}
we can define a discrete Laplacian operator that acts linearly on the vertex function $\mathbf{f}$ of the mesh.

Let $\mathbf{f}\in\mathbb{R}^n$ be a function defined on the $n$ vertices of the mesh $\mathcal{M}$. According to the concept of an approach to the Dirichlet energy in \eqref{approach}, a Laplacian matrix acting on the function $\mathbf{f}$ is defined based on differences along the edges:  
\begin{align}\label{dualL}
(\mathbf{L}\mathbf{f})_i=\sum_{\mathbf{v}_i, \mathbf{v}_j\in\mathcal{M}}\mathit{w}_{ij}(f_i-f_j),
\end{align}
where
\begin{align}\label{dualL1}
\mathbf{L}_{ij}=
\begin{cases}
w_{ij}& ,\text{for $i\neq j$},\\
-\sum_{\mathbf{v}_i, \mathbf{v}_j\in\mathcal{M}}w_{ij}& ,\text{otherwise}.
\end{cases}
\end{align}
where $w_{ij}\neq 0$ only if vertices $\mathbf{v}_i$ and $\mathbf{v}_j$ are connected by an edge in $\mathcal{M}$ and a Laplacian vanishes on a constant function $\mathbf{f}\equiv \mathbf{1}$, i.e., vector of ones in $\mathbb{R}^n$.

\begin{remark}
In \cite{wardetzky2007}, Wardetzky et al. considered triangle meshes and discussed several desirable properties such as symmetry, locality, linear precision, the maximum principle, positive semidefiniteness, and the convergence of Dirichlet problem. In particular, they proved an important theoretical restriction that the discrete Laplacian cannot satisfy all natural properties, and then demonstrated that the discrete Laplacian defined in \eqref{dualL} and \eqref{dualL1} can have some of the natural properties of the smooth Laplacian if certain types of meshes are considered. This limitation provides a categorization of the existing literature and explains the large number of existing discrete Laplacians. Since not all desired properties can be satisfied at the same time, the design of a discrete Laplacian depends on the application at hand in order to satisfy the particular needs of a specific problem.
\end{remark}

%--------------------------------------------------------------------------
\subsubsection{Primal topology and dual topology }
Given a triangular or tetrahedral mesh $\mathcal{M}$, we denote the vertices and their positions in the local frame of reference by $\mathbf{v}_i$ (0-simplex). Furthermore, the positions of the vertices in $\mathcal{M}$ are given in the form of vectors $\mathbf{V}$ in $\mathbb{R}^{n\times d}$, so the columns are $n-$vectors consisting of the components of all vertices and the rows are the positions of the vertices in $\mathbb{R}^d$. Taking $d=3$, i.e., $\mathbf{v}_m=[v_m^1, v_m^2, v_m^3]\in\mathbb{R}^{1\times 3}$ for example, 
\begin{align*}
\mathbf{V}=
\begin{bmatrix}
\mathbf{v}_1\\
\mathbf{v}_2\\
\vdots\\
\mathbf{v}_n
\end{bmatrix}=
\begin{bmatrix}
v_1^1&v_1^2&v_1^3\\
v_2^1&v_2^2&v_2^3\\
\vdots&\vdots&\vdots\\
v_n^1&v_n^2&v_n^3\\
\end{bmatrix}
\end{align*}
Two vertice and a straight line define an edge $[\mathbf{v}_i, \mathbf{v}_j]:=\mathbf{e}$ (1-simplex), which consists of the distance and direction from $\mathbf{v}_i$ to $\mathbf{v}_j$. Multiple edges construct a surface which reduces to the union of triangles $[\mathbf{v}_i, \mathbf{v}_j, \mathbf{v}_k]:= \mathrm{T}$ (2-simplex). Assembling multiple surfaces, we can construct volumes and expressed as the union of  tetrahedra by $[\mathbf{v}_i, \mathbf{v}_j, \mathbf{v}_k, \mathbf{v}_l]:=\tau$ (3-simplex).For more discussion of n-complex, we refer to \cite[Chap. 6.3]{jing2007fund}. Each of these topologies is described as primal topologies.

In addition to the primal topology that describes the properties of the original space, the dual topology also reflects the properties of the original space in its dual space. In the following, the Hodge dual will be introduced to define the dual topology. Let $\mathbf{e}_1, \cdots, \mathbf{e}_m$ be an orthonormal basis of $\mathbb{R}^m$, the $k-$fold exterior product of $\mathbb{R}^m$, $\Lambda^k\left(\mathbb{R}^m\right)$ is spanned by an orthonormal basis
\begin{align*}
\mathbf{e}_{i_1}\wedge\cdots\wedge \mathbf{e}_{i_k},\quad 1\leq i_1<\cdots<i_k\leq m.
\end{align*}
\begin{definition}\cite[Chap. 3.3]{jost2011}
Suppose $\mathbb{R}^m$ carry an positive orientation i.e., 
\begin{align*}
\det\left(\left[\mathbf{e}_1\cdots\mathbf{e}_m\right]\right)=1,\quad \mathbf{e}_k=[0, \cdots, \underset{\substack{\uparrow\\\mathclap{k-th}}}{1}, \cdots, 0]^{\intercal}.
\end{align*}
The Hodge star (dual) operator is defined by
\begin{align}\label{hodge}
*:&\Lambda^k\left(\mathbb{R}^m\right)\longrightarrow\Lambda^{m-k}\left(\mathbb{R}^{m}\right),\quad 0\leq k\leq m,\nonumber\\
&*\left(\mathbf{e}_{i_1}\wedge\cdots\wedge \mathbf{e}_{i_k}\right)=\pm\mathbf{e}_{j_1}\wedge\cdots\wedge \mathbf{e}_{j_{m-k}},
\end{align}
where $\left[\mathbf{e}_{i_1}, \cdots, \mathbf{e}_{i_k}, \mathbf{e}_{j_1}, \cdots, \mathbf{e}_{j_{m-k}}\right]$ has the positive (negative) orientation. In particular,
$
*\left(\mathbf{e}_{1}\wedge\cdots\wedge \mathbf{e}_{m}\right)=1$ and $*\left(1\right)=\mathbf{e}_{1}\wedge\cdots\wedge \mathbf{e}_{m}.
$
\end{definition}

\begin{example}
Suppose $\mathbb{R}^5$ carry the standard positive orientation $\left[\mathbf{e}_1, \mathbf{e}_2, \mathbf{e}_3, \mathbf{e}_4, \mathbf{e}_5\right]$, the Hodge operator acting on $\mathbf{e}_3$, $\mathbf{e}_1\wedge\mathbf{e}_3$ and $\mathbf{e}_1\wedge\mathbf{e}_2\wedge\mathbf{e}_4$ are given by
\begin{equation*}
*\left(\mathbf{e}_3\right)=\mathbf{e}_1\wedge\mathbf{e}_2\wedge\mathbf{e}_4\wedge\mathbf{e}_5,\quad *\left(\mathbf{e}_1\wedge\mathbf{e}_3\right)=-\mathbf{e}_2\wedge\mathbf{e}_4\wedge\mathbf{e}_5,\quad*\left(\mathbf{e}_1\wedge\mathbf{e}_2\wedge\mathbf{e}_4\right)=-\mathbf{e}_3\wedge\mathbf{e}_5.
\end{equation*}
\end{example}

The geometric interpretation of the Hodge star operation in \eqref{hodge} can be obtained from the subspace $U: =\mathrm{span}\{\mathbf{e}_{i_1}, \cdots, \mathbf{e}_{i_k}\}$ of $\mathbb{R}^m$ and its orthogonal subspace $U^{\perp}$. This can be achieved by recognizing the fundamental matrix $\left[\mathbf{e}_{i_1}, \cdots,\mathbf{e}_{i _k}\right]$ as $\mathbf{e}_{i_1}\wedge\cdots\wedge\mathbf{e }_{i _k}$ by the Pl\"{u}cker embedding. Based on the geometric interpretation, the dual topology can be applied to discrete meshes. Taking the tetrahedral mesh $\mathcal{M}$ as an example, $*[\mathbf{v}_i, \mathbf{v}_j]$ is the face dual to the edge and $*[\mathbf{v}_i, \mathbf{v}_j, \mathbf{v}_k, \mathbf{v}_l]$ is the circumcenter (vertex) dual to the tetrahedron $[\mathbf{v}_i, \mathbf{v}_j, \mathbf{v}_k, \mathbf{v}_l]$. 
The relative measures of an element and its dual are denoted by $|\cdot|$. For elements in the mesh, $|[\mathbf{v}_i, \mathbf{v}_j]|$ is the length of the edge, $|[\mathbf{v}_i, \mathbf{v}_j, \mathbf{v}_k]|$ is the area of the triangle and $|[\mathbf{v}_i, \mathbf{v}_j, \mathbf{v}_k, \mathbf{v}_l]|$ is the volume of the tetrahedron. Similarly, the measures of dual elements are the same, e.g., $|*[\mathbf{v}_i, \mathbf{v}_j]|$ is the length of edge dual to $[\mathbf{v}_i, \mathbf{v}_j]$ in a triangle mesh, resp. the area of face dual to $[\mathbf{v}_i, \mathbf{v}_j]$ in a tetrahedral mesh.

\begin{figure}
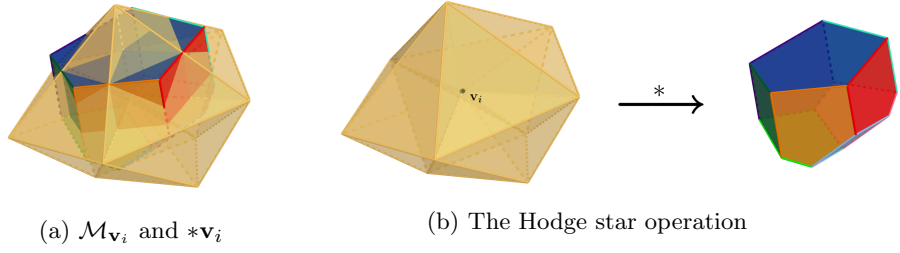

\begin{subfigure}[h]{0.4\textwidth}
\centering
   \includegraphics[width=.65\linewidth]{figures/Pri_Dual_cell.pdf}
   \caption{$\mathcal{M}_{\mathbf{v}_i}$ and $*\mathbf{v}_i$}
   \label{fig:pri_dual}
\end{subfigure}
\begin{subfigure}[h]{0.5\textwidth}
\centering
   \includegraphics[width=1.15\linewidth]{figures/Pri_Hodge_Dual.pdf}
   \caption{The Hodge star operation}
   \label{fig:pri_hodge_dual}
\end{subfigure}
\caption{(a) The one-ring neighborhood $\mathcal{M}_{\mathbf{v}_i}$ and dual cell $*\mathbf{v}_i$ of the vertex $\mathbf{v}_i$. (b) The Hodge star $*$ maps the vertex $\mathbf{v}_i$ to its dual $*\mathbf{v}_i$}
\end{figure}

%--------------------------------------------------------------------------
\section{Primal and dual construction of the discrete Laplacian}
In this section, we review the primal and dual construction of the discrete Laplacian of triangle and tetrahedral meshes.
For each construction, we explain only the form of the weights $w_{ij}$ in \eqref{dualL1} and the notation therein; 
the detailed derivations are not discussed here, but we briefly list the references for the reader's information.

%--------------------------------------------------------------------------
\subsection{Primal discrete Laplacian}
There are a number of ways to construct the primal discrete Laplacian, for example, by constructing piecewise linear basis functions on a triangular or tetrahedral mesh in the finite element method \cite{dziuk1988}, by directly calculating the Dirichlet energy of the piecewise linear function \cite{pinkall1993,gu2003volumetric} or consider immersing the mesh into higher dimensions and then using the fact that  the Laplacian is proportional to the gradient of the area functional to define the vertex action \cite{crane2019n}. Either way, the result of the weight $w_{ij}$ in \eqref{dualL1} is the sum of the contributions of all incident triangles or tetrahedra:
\begin{itemize}
\item Triangle mesh $\mathcal{M}$:\\
The weight to the edge $[\mathbf{v}_i, \mathbf{v}_j]$ is given by
\begin{subequations}\label{triangle}
\begin{align}
w_{ij}=\sum_{[\mathbf{v}_i, \mathbf{v}_j, \mathbf{v}_k]\in\mathcal{M}}w_{ijk}.
\end{align}
The contribution of each incident triangles with common edge $[\mathbf{v}_i, \mathbf{v}_j]$ is
\begin{align}
w_{ijk}:=-\frac{\mathbf{n}_i\cdot\mathbf{n}_j}{4|[\mathbf{v}_i, \mathbf{v}_j, \mathbf{v}_k]|}=\frac{1}{2}\cot\theta_k.
\end{align}
The symbol $\mathbf{n}_i:= \mathbf{s}\times (\mathbf{v}_k-\mathbf{v}_j)$ denotes the conormal vector in the plane of the triangle $[\mathbf{v}_i, \mathbf{v}_j, \mathbf{v}_k]$ orthogonal to the edge $[\mathbf{v}_j, \mathbf{v}_k]$, where $\mathbf{s}$ is the unit surface normal vector of the triangle $[\mathbf{v}_i, \mathbf{v}_j, \mathbf{v}_k]$. The angle at the vertex $\mathbf{v}_k$ is denoted by $\theta_k$.
\end{subequations}

\begin{figure}[t]
  \centering
  % the following command controls the width of the embedded PS file
  % (relative to the width of the current column)
  \includegraphics[width=.6\linewidth]{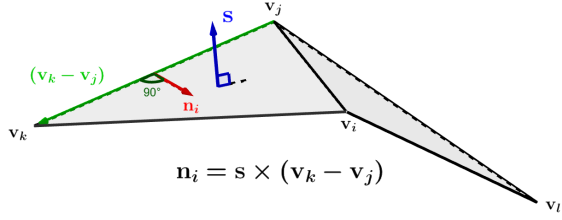}
  % replacing the above command with the one below will explicitly set
  % the bounding box of the PS figure to the rectangle (xl,yl),(xh,yh).
  % It will also prevent LaTeX from reading the PS file to determine
  % the bounding box (i.e., it will speed up the compilation process)
  % \includegraphics[width=.95\linewidth, bb=39 696 126 756]{sampleFig}
  %
  %\parbox[t]{.9\columnwidth}{\relax
     %      For all figures please keep in mind that you \textbf{must not}
        %   use images with transparent background! 
           %}
  \caption{\label{fig:Cotangent}
           The unit surface normal vector $\mathbf{s}$ and the conormal vector $\mathbf{n}_i$ on the triangle $[\mathbf{v}_i, \mathbf{v}_j, \mathbf{v}_k]$.}
\end{figure}

\item Tetrahedral mesh $\mathcal{M}$:\\
The weight to the edge $[\mathbf{v}_i, \mathbf{v}_j]$ is given by
\begin{subequations}
\begin{align}\label{primalL2}
w_{ij}=\sum_{[\mathbf{v}_i, \mathbf{v}_j, \mathbf{v}_k, \mathbf{v}_l]\in\mathcal{M}}w_{ijkl}.
\end{align}
The contribution of each incident triangles with common edge $[\mathbf{v}_i, \mathbf{v}_j]$ is
\begin{align}\label{primalL3}
w_{ijkl}:=-\frac{\mathbf{s}_i\cdot\mathbf{s}_j}{18|[\mathbf{v}_i, \mathbf{v}_j, \mathbf{v}_k, \mathbf{v}_l]|}=\frac{1}{6}|[\mathbf{v}_k, \mathbf{v}_l]|\cot\theta_{kl}^{ij}.
\end{align}
See the description in Figure~\ref{fig:Cotangent}(a). The surface normal vector $\mathbf{s}_i$ of the triangle $[\mathbf{v}_j, \mathbf{v}_k, \mathbf{v}_l]$ is chosen in the outer direction. The angle $\theta_{kl}^{ij}$ is the dihedral angle of the edge $[\mathbf{v}_k, \mathbf{v}_l]$ opposite the edge $[\mathbf{v}_i, \mathbf{v}_j]$.
%\begin{figure}[htb]
  %\centering
  % the following command controls the width of the embedded PS file
  % (relative to the width of the current column)
  %\includegraphics[width=1.0\linewidth]{figures/Cotangent_weight_tetra0.pdf}
  % replacing the above command with the one below will explicitly set
  % the bounding box of the PS figure to the rectangle (xl,yl),(xh,yh).
  % It will also prevent LaTeX from reading the PS file to determine
  % the bounding box (i.e., it will speed up the compilation process)
  % \includegraphics[width=.95\linewidth, bb=39 696 126 756]{sampleFig}
  %
  %\parbox[t]{.9\columnwidth}{\relax
     %      For all figures please keep in mind that you \textbf{must not}
        %   use images with transparent background! 
           %}
  %
  %\caption{\label{fig:Cotangent}
     %      Primal construction of discrete Laplacian on a tetrahedron $[\mathrm{v}_i, \mathrm{v}_j, \mathrm{v}_k, \mathrm{v}_l]$.}
%\end{figure}
\end{subequations}
\end{itemize}

\begin{figure}
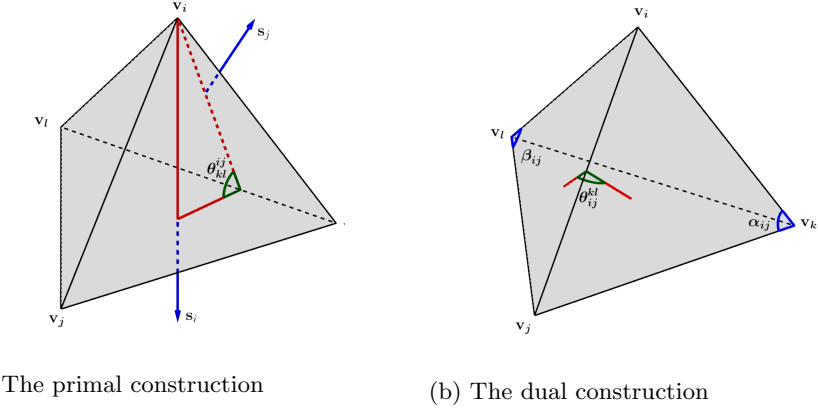

\begin{subfigure}[h]{0.45\textwidth}
   \includegraphics[width=1.3\linewidth]{figures/Cotangent_weight_tetra0.pdf}
   \caption{The primal construction}
   %\label{fig:pri_dual}
\end{subfigure}
\begin{subfigure}[h]{0.45\textwidth}
   \includegraphics[width=1.3\linewidth]{figures/Cotangent_weight_tetra1.pdf}
   \caption{The dual construction}
   %\label{fig:pri_hodge_dual}
\end{subfigure}
\caption{\label{fig:Cotangent}(a) The dihedral angle $\theta_{kl}^{ij}$ in the primal construction of the weight $w_{ijkl}$ to the edge $[\mathbf{v}_i, \mathbf{v}_j]$. (b) The dihedral angle $\theta_{ij}^{kl}$ and triangular angles $\alpha_{ij}$, $\beta_{ij}$  in the dual construction of the weight $w_{ijkl}$ to the edge $[\mathbf{v}_i, \mathbf{v}_j]$.}
\end{figure}

%--------------------------------------------------------------------------
\subsection{Dual discrete Laplacian}
According to the notion of approximation to the Dirichlet energy in \eqref{approach}, the discrete Laplacian $(\mathbf{L}\mathbf{f})_i$ represents the integrated Laplacian; that is, it defines the Laplacian operator of a vertex $\mathbf{v}_i$ in the integral (average) sense. The correspondence of the Dirichlet energy and the Laplacian operator in the continuous and discrete cases can be found in the table preceding Section 2.2. of \cite{wardetzky2007discrete}. Thus, different approximations to the domain around a vertex (one-ring neighbourhood) yield different constructions of the discrete Laplacian. In contrast to the direct modelling of the primal construction to Laplacian, the dual approach uses inaccurate discrete data from the Dirichlet energy minimisation process as an approximation to the Laplcian. 

To correlate the dual approach with the integral (Dirichlet energy), the classical Stokes' theorem is a good starting point. More precisely, the generalized Stokes' theorem relates the integral of a differential $(n-1)$-form $\omega$ on the boundary $\partial\mathcal{M}$ to the integral of a $n$-form $d\omega$, which is the exterior derivative of $\omega$, on a $n$-dimensional manifold $\mathcal{M}$. For the thorough introduction to the discrete exterior calculus (DEC), the works by Hirani \cite{hirani2003discrete} and Desbrun et al. \cite{desbrun2005discrete} present the history and development in the discretization of differential operators through the exterior calculus. In \cite{hirani2003discrete,desbrun2006discrete}, the DEC is to assign the region associated with vertex to be its dual cell and assume the dual graph is orthogonal and immersed  in order to be consistent with the Hodge star operator.

For triangle meshes, Pinkall and Polthier \cite{pinkall1993} considered the conjugation i.e., the Hodge star (dual) operation, on discrete surface and present the conjugation algorithm for computing the conjugate minimal surface. Most importantly, they proved that the minimality condition for the discrete Dirichlet energy of initial surfaces is equivalent to the closedness condition for the dual one-form of conjugate surfaces in \cite[Theorem 11.]{pinkall1993}. In other words, the primal and dual Laplacians are equivalent in the planar case.

For tetrahedral meshes, Alexa et al. \cite{alexa2020properties} applied Stokes theorem to the integrated Laplacian by designating the region associated with the vertex $\mathbf{v}_i$ as its dual unit $*\mathbf{v}_i$ under the assumption that the dual graph is orthogonal and immersed and thus deriving the dual discrete Laplacian, whose edge weights are different from that of the primal discrete Laplacian. Let $f$ be the piecewise linear approximation to the exact solution, the derivation of $(\mathbf{L}\mathbf{f})_i$ is as follows,

\begin{align*}
(\mathbf{L}\mathbf{f})_i
=&\int_{*\mathbf{v}_i}\Delta f=\int_{\partial*\mathbf{v}_i}\mathbf{n}^{\top}\nabla f\\
=&\sum_{[\mathbf{v}_i, \mathbf{v}_j]\in\mathcal{M}}\int_{*[\mathbf{v}_i, \mathbf{v}_j]}\frac{\mathbf{v}_j-\mathbf{v}_i}{|[\mathbf{v}_i, \mathbf{v}_j]|}\nabla f,\quad ((\mathbf{v}_j-\mathbf{v}_i)\nabla f=(f_j-f_i))\\
=&\sum_{[\mathbf{v}_i, \mathbf{v}_j]\in\mathcal{M}}\frac{|*[\mathbf{v}_i, \mathbf{v}_j]|}{|[\mathbf{v}_i, \mathbf{v}_j]|}(f_j-f_i)\\
:=&\sum_{[\mathbf{v}_i, \mathbf{v}_j]\in\mathcal{M}}w_{ij}(f_j-f_i).
\end{align*}
The per-tetrahedron accumulation formula of $w_{ij}$ is defined by
\begin{subequations}\label{dualL2_3}
\begin{align}\label{dualL2}
w_{ij}=\sum_{[\mathbf{v}_i, \mathbf{v}_j, \mathbf{v}_k, \mathbf{v}_l]}w_{ijkl}+w_{ijlk}
\end{align}
where $w_{ijkl}$ can be explicitly formulated by the dihedral angle $\theta_{ij}^{kl}$ on the edge $[\mathbf{v}_i, \mathbf{v}_j]$ and the triangular interior angles $\alpha_{ij}$ and $\beta_{ij}$ incident on the opposite edge $[\mathbf{v}_i, \mathbf{v}_j]$ see Figure~\ref{fig:Cotangent} (b),
\begin{align}\label{dualL3}
w_{ijkl}=\frac{|[\mathbf{v}_i, \mathbf{v}_j]|}{8}\cot\theta_{ij}^{kl}\left(\frac{2\cot\alpha_{ij}\cot\beta_{ij}}{\cos\theta_{ij}^{kl}}-(\cot^2\alpha_{ij}+\cot^2\beta_{ij})\right).
\end{align}
\end{subequations}

\begin{remark}
We revisit these two constructions of Laplacian on the tetrahedral mesh $\mathcal{M}$. Although each tetrahedron $\tau\in\mathcal{M}$ is a three-dimensional object, the vertex function $\mathbf{f}$ acts only on the vertices of $\tau$, while the action on the interior of $\tau$ is entirely governed by the action on the vertices. Thus, the Dirichlet energy of a smooth function $f$ on $\tau$ should be considered as the problem of embedding from a two-dimensional surface, say $\partial \tau$, into a three-dimensional space $\mathbb{R}^3$. 
\end{remark}
Thus $\mathcal{M}$ is still represented as a tetrahedral mesh and we consider it to be the two-dimensional interface that divides a compact 3-manifold into multiple tetrahedral subregions unless it is otherwise stated,
$
\mathcal{M}\equiv\bigcup_{\tau\in \mathcal{M}}\partial \tau.
$
%\begin{figure}[htb]
  %\centering
  % the following command controls the width of the embedded PS file
  % (relative to the width of the current column)
  %\includegraphics[width=1.0\linewidth]{figures/Cotangent_weight_tetra1.pdf}
  % replacing the above command with the one below will explicitly set
  % the bounding box of the PS figure to the rectangle (xl,yl),(xh,yh).
  % It will also prevent LaTeX from reading the PS file to determine
  % the bounding box (i.e., it will speed up the compilation process)
  % \includegraphics[width=.95\linewidth, bb=39 696 126 756]{sampleFig}
  %
  %\parbox[t]{.9\columnwidth}{\relax
     %      For all figures please keep in mind that you \textbf{must not}
        %   use images with transparent background! 
           %}
  %
  %\caption{\label{fig:Cotangent1}
     %      Dual construction of discrete Laplacian on a tetrahedron $[\mathrm{v}_i, \mathrm{v}_j, \mathrm{v}_k, \mathrm{v}_l]$.}
%\end{figure}

%--------------------------------------------------------------------------
\section{Discrete Construction of Mean curvature}
In differential geometry, given an orientable surface $\mathcal{M}$ smoothly immersed in $\mathbb{R}^3$, we can study its local shape  by means of the shape operator (Weingarten map), assigning a unit normal vector at the point $x\in\mathcal{M}$.
\iffalse
\begin{align*}
S_{\mu}:T_{x}\mathcal{M}\to T_{\mu}\mathbb{S}^2\cong T_{x}\mathcal{M},
\end{align*}
where $\mu$ is a unit normal vector to $\mathcal{M}$ at point $x$. The linear map $S_{\mu}$ is called the shape operator
\fi 
Furthermore, the shape operator is the second-order invariant that completely determines the surface $\mathcal{M}$ in $\mathbb{R}^3$. In particular, the Laplace operator defined on $\mathcal{M}$ and the quantity defined by the shape operator are related as follows,
\begin{align}\label{second}
\Delta_{\mathcal{M}} f=\vec{H}
\end{align}
where $f$ is the coordinate map embedding $\mathcal{M}$ into $\mathbb{R}^{3}$. A more complete argument for the connection between \eqref{eq:EL} and \eqref{DiriEL} and the derivation of \eqref{second} can be found in \cite[Chap. 8]{jost2011} and \cite[Appendix A]{ecker2012regularity}.

Consider a triangle mesh $\mathcal{M}\hookrightarrow\mathbb{R}^3$ and the discrete Laplacian defined by \eqref{dualL}, \eqref{dualL1} and \eqref{triangle}, it is reasonable to construct the discrete mean curvature on $\mathcal{M}$ and is expected to satisfy the discrete version of \eqref{second}.

%--------------------------------------------------------------------------
\subsection{Discrete mean curvature on vertex}\label{sec:4}
 Suppose that a smooth surface moves in the direction of the surface normal, the mean curvature vector can represent a measure of the variation in surface area relative to nearby surfaces. Let $\mathrm{T}$ be a triangle in the triangle mesh $\mathcal{M}$ with vertices $\mathbf{v}_i$ and vertex angles $\alpha_i$, $\mathrm{Area}(\mathrm{T})$
\iffalse
\footnote{Given a triangle $[A, B, C]$. Let $A=\alpha_1$, $B=\alpha_2$ and $C=\alpha_3$ whose opposite edges are $a=|\mathbf{v}_2-\mathbf{v}_3|$, $b=|\mathbf{v}_3-\mathbf{v}_1|$ and $c=|\mathbf{v}_1-\mathbf{v}_2|$. \eqref{areaform} can be derived by the following equalities,
\begin{align*}
&\cos A=\frac{b^2+c^2-a^2}{2bc},\quad
\iffalse \cos B=\frac{a^2+c^2-b^2}{2ac},\ \cos C=\frac{a^2+b^2-c^2}{2ab}\fi
\triangle=\frac{1}{2}bc\sin A\\
\iffalse=\frac{1}{2}ac\sin B=\frac{1}{2}ab\sin C\fi
&\cot A=\frac{b^2+c^2-a^2}{4\triangle},\quad
\iffalse \cot B=\frac{a^2+c^2-b^2}{4\triangle},\ \cot C=\frac{a^2+b^2-c^2}{4\triangle}\fi
\triangle=\sqrt{s(s-a)(s-b)(s-c)},\ s=\frac{a+b+c}{2}.
\end{align*}} 
\fi
can be expressed as
\begin{align}\label{areaform}
\mathrm{Area}(\mathrm{T})=\frac{1}{4}\sum_{i=1}^3\cot\alpha_i|\mathbf{v}_{i+1}-\mathbf{v}_{i-1}|^2
\end{align}
For detailed discussions and derivations of curvatures, harmonic maps, minimal surfaces, etc. in the discrete case, the readers can consult \cite{pinkall1993, polthier2002, gu2008} or reference therein.

\iffalse
\begin{remark}Let $A=\alpha_1$, $B=\alpha_2$ and $C=\alpha_3$ whose opposite edges are $a=|\mathbf{v}_2-\mathbf{v}_3|$, $b=|\mathbf{v}_3-\mathbf{v}_1|$ and $c=|\mathbf{v}_1-\mathbf{v}_2|$, respectively. Therefore, \eqref{areaform} can be derived by the following equalities,
\begin{align*}
&\cos A=\frac{b^2+c^2-a^2}{2bc},\ \cos B=\frac{a^2+c^2-b^2}{2ac},\ \cos C=\frac{a^2+b^2-c^2}{2ab}\\
&\triangle=\frac{1}{2}bc\sin A=\frac{1}{2}ac\sin B=\frac{1}{2}ab\sin C\\
&\cot A=\frac{b^2+c^2-a^2}{4\triangle},\ \cot B=\frac{a^2+c^2-b^2}{4\triangle},\ \cot C=\frac{a^2+b^2-c^2}{4\triangle}\\
&\triangle=\sqrt{s(s-a)(s-b)(s-c)},\ s=\frac{a+b+c}{2}
\end{align*}
\end{remark}
\fi

\begin{lemma}
Let $\mathbf{v}$ be an interior vertex of $\mathcal{M}$ and $\mathcal{M}_\mathbf{v}$ be the triangles forming a one--ring neighborhood of $\mathbf{v}$. Then the gradient of $\mathrm{Area}(\mathcal{M}_{\mathbf{v}})$ with respect to variation of vertices can be written in the cotangent formula
\begin{align}\label{gradarea}
\nabla\mathrm{Area}(\mathcal{M}_{\mathbf{v}})=\frac{1}{2}\sum_{\mathbf{v}_i \in \mathcal{M}_{\mathbf{v}}
%\substack{\text{vertices $\mathbf{v}_i$}\\ \text{adjacent to $\mathbf{v}$}}
}\left(\cot\alpha_i+\cot\beta_i\right)(\mathbf{v}-\mathbf{v}_i),
\end{align}
where $\alpha_i$ and $\beta_i$ are the opposite angles relative to the edge $[\mathbf{v}_i, \mathbf{v}]$.
\end{lemma}
\begin{remark}\label{rmk:normal}
In the discrete case, the normal vector is not uniquely determined not only on the edges but also at the vertices. According to the calculus of variations on the area functional,  the area varies rapidly in the direction of the surface normal, so we can choose the normal vector as the gradient of area. In particular, for planar regions, the area around each point neither increases nor decreases; that is, there are no bulges or depressions at each point, so the normal vector is the same for each point. Renormalizing the normal vector by its local area, the unit normal vector of each point on the plane region can be defined and parallel to each other.
\end{remark}
\begin{definition}\label{discreteVertex}
The discrete mean curvature at the vertex $\mathbf{v}\in\mathcal{M}_{\mathbf{v}}$ is a vector--valued quantity
\begin{align}\label{discreteH}
\vec{H}_{\mathbf{v}}:=\nabla\mathrm{Area}(\mathcal{M}_{\mathbf{v}})
=\frac{1}{2}\sum_{\mathbf{v}_i \in \mathcal{M}_{\mathbf{v}}
%\substack{\text{vertices $\mathbf{v}_i$}\\ \text{adjacent to $\mathbf{v}$}}
}\left(\cot\alpha_i+\cot\beta_i\right)(\mathbf{v}-\mathbf{v}_i)
=\mathbf{L}_{\mathbf{v}\mathbf{v}_i}(\mathbf{v}-\mathbf{v}_i).
\end{align}
where $\mathbf{L}_{\mathbf{v}\mathbf{v}_i}$ is the discrete cotangent--weight Laplacian. 
\end{definition}

%--------------------------------------------------------------------------
\subsection{Discrete mean curvature on edge}
When processing discrete surfaces, mean curvature occurs not only at vertices but also along edges. Since vertices are intersections of edges, an intuitive idea is to define discrete mean curvature along the edges and the mean curvature of the edges lying on the plane must be zero. The discussions of discrete mean curvature along edges can be found in \cite[pp.85-86]{polthier2002} and \cite[Sec.4.3]{sullivan2008} . We follow the notation and definition of the latter article.

\begin{definition}\cite{sullivan2008}\label{discreteEdge}
Suppose $\mathrm{T}_k:= [\mathbf{v}_i, \mathbf{v}_j, \mathbf{v}_k]$ and $\mathrm{T}_l:= [\mathbf{v}_i, \mathbf{v}_j, \mathbf{v}_l]$ are two non-coplanar triangles with a common edge $\mathbf{e}:= [\mathbf{v}_i, \mathbf{v}_j]$, see \textup{Figure~\ref{fig:DiscreteMean} (a)}, then the discrete mean curvature vector at $\mathbf{e}$ is defined by the total mean curvature over any surface patch $\Omega$ in the star region $\mathrm{Star}(\mathbf{e})=\mathrm{T}_k\cup \mathrm{T}_l$ containing $\mathbf{e}$,
\begin{align}
\label{discreteE1}
-2\vec{H}_{\mathbf{e}}
\equiv&\iint_{\Omega}\vec{H}dA=\iint_{\Omega}H\mu dA=-\oint_{\partial \Omega}\mu\times d\mathbf{x}\nonumber\\
=&\oint_{\partial \Omega}\eta ds=J_k \mathbf{e}-J_l \mathbf{e},\quad \eta:=t\times\mu
\end{align}
where $\mu$ and $t$ is the unit normal and unit tangent vectors on surface patch $\Omega$, respectively, and $d\mathbf{x}=t\ ds$ is the vector line element along $\partial \Omega$. The operator $J_k$ $(J_l)$
\footnote{
\iffalse
The complex structure on $\mathbb{R}^2$ is a linear map
$
J:\mathbb{R}^2\longrightarrow\mathbb{R}^2$, $(x^1, x^2)\mapsto(-x^2, x^1)
$ 
i.e., counterclockwise rotation $\pi/2$. 
\fi
If $\mathcal{M}$ is a two-dimensional vector space, then a complex structure on $\mathcal{M}$ can be defined as
$
J:\mathcal{M}\longrightarrow\mathcal{M}$, rotating counterclockwise rotation by $\pi/2$, so $J^2=-\mathrm{Id}_{\mathcal{M}}
$
where $J^2=J\circ J$ and $\mathrm{Id}_{\mathcal{M}}$ is the identity map on $\mathcal{M}$.
}
 is a complex structure in the plane of the triangle $\mathrm{T}_k$ $(\mathrm{T}_l)$ by rotating $90^{\circ}$. Let $\theta_{\mathbf{e}}$ be the dihedral angle along the edge $\mathbf{e}$ enclosed by triangles $\mathrm{T}_h$ and $\mathrm{T}_l$, the mean curvatures are given by
\begin{align}
\label{discreteE2}
|\vec{H}_{\mathbf{e}}|=\frac{|J_l \mathbf{e}-J_k \mathbf{e}|}{2}=\sin((\pi-\theta_{\mathbf{e}})/2)|\mathbf{e}|.
\end{align}
In particular, $\theta_{\mathbf{e}}$ is set to $0$ for describing the mean curvature (normal) vector along the boundary edge $\mathbf{e}$, since $\vec{H}_{\mathbf{e}}=J \mathbf{e}$ and $|\vec{H}_{\mathbf{e}}|=|\mathbf{e}|$ by \eqref{discreteE1}, \eqref{discreteE2} and the condition $\theta_{\mathbf{e}}\equiv 0\Longleftrightarrow J_k=-J_l\equiv J$.
\end{definition}

\begin{remark}\label{rmkVE}
The discrete mean curvature defined in Definition~\ref{discreteVertex} and Definition~\ref{discreteEdge} is equivalent, because the accumulation formula for the vertices is shown below:
\begin{align*}
\vec{H}_{\mathbf{v}}
:=\sum_{\mathbf{v}_i \in \mathcal{M}_{\mathbf{v}}
%\substack{\text{vertices $\mathbf{v}_i$}\\ \text{adjacent to $\mathbf{v}$}}
}-\vec{H}_{[\mathbf{v}, \mathbf{v}_i]}=&\frac{1}{2}\sum_{\mathbf{v}_i \in \mathcal{M}_{\mathbf{v}}} J_i(\mathbf{v}-\mathbf{v}_{i})=\mathbf{L}_{\mathbf{v}\mathbf{v}_i}(\mathbf{v}-\mathbf{v}_i).
%=&\textcolor{cyan}{\frac{1}{2}\sum_i(\cot\alpha_i+\cot\beta_i)(\mathrm{v}-\mathrm{v}_i),}
\end{align*}
where the derivation of last equality can be found in \cite[pp.24]{pinkall1993} \cite[Lemma 3.6]{wardetzky2008}. Most importantly, the discrete mean curvature vector in \eqref{discreteH} is consistent with the mean curvature vector in the smoothed case in \eqref{second}.
\end{remark}

In the conclusion of this section, we give the corresponding discrete forms of the tangential smooth Laplace--Beltrami operator and the normal smooth Laplace--Beltrami operator for a given triangular mesh $\mathcal{M}$, considered as a hypersurface in $\mathbb{R}^3$. Let $\mathbf{L}=[\mathbf{L}_{ij}]=\left[\frac{1}{2}\left(\cot\widehat{\alpha}_{ij}+\cot\widehat{\beta}_{ij}\right)\right]$ is a symbolic matrix induced by the triangle mesh $\mathcal{M}$.
\begin{itemize}\label{sec:End4}
\item The tangential discrete Laplacian of $\mathcal{M}$:
\begin{align*}
\mathbf{L}^{\top}(\mathcal{M})=[\mathbf{L}^{\top}_{ij}](\mathcal{M})=\left[\frac{1}{2}\left(\cot\widehat{\alpha}_{ij}+\cot\widehat{\beta}_{ij}\right)^{\top}\right](\mathcal{M}):=\left[\frac{1}{2}\left(\cot\alpha_{ij}+\cot\beta_{ij}\right)\right].
\end{align*}
In other words, it is just the standard cotangent--weighted discrete Laplacian of the triangular mesh $\mathcal{M}$ derived in \cite{pinkall1993,polthier2002,gu2008}.
\item The normal discrete Laplacian of $\mathcal{M}$:
\begin{align*}
\mathbf{L}^{\perp}(\mathcal{M})=[\mathbf{L}^{\perp}_{ij}](\mathcal{M})=\left[\frac{1}{2}\left(\cot\widehat{\alpha}_{ij}+\cot\widehat{\beta}_{ij}\right)^{\perp}\right](\mathcal{M}):=\left[\sin\left(\frac{\pi-\theta_{ij}}{2}\right)\right].
\end{align*}
Clearly, $\mathbf{L}^{\perp}(\mathcal{M}_{\mathbf{v}})\equiv 0$ implies that the surface $\mathcal{M}_{\mathbf{v}}$ is not curved at the edges; that is, $\mathcal{M}_{\mathbf{v}}$ is a locally minimal surface, a planar domain. This idea was first implicit in \cite{pinkall1993}, while precise geometric perspectives and arguments were explicitly given in \cite{polthier2002,sullivan2008}.
\end{itemize}

%--------------------------------------------------------------------------
\section{Geometric aspects of the dual construction}\label{sec:5}
In contrast to the Laplacian on triangle meshes, the primal and dual Laplacians on tetrahedral meshes are formulated in different ways. 
 It is important to determine which structure (primal or dual) satisfies the discrete form of \eqref{second} when we want to generalize the argument to higher dimensions, such as $\mathbb{R}^4$, or apply some structures from $\mathbb{R}^4$ to study the objects in $\mathbb{R}^3$.

Since we live in three dimensions, we cannot directly generalize the arguments in \cite{pinkall1993,polthier2002,gu2008} to check the equation \eqref{discreteH}. Even so, we can compare which structure gives the optimal approximation to the Laplacian in $\mathbb{R}^3$ if it involves a classical immersion from $\mathbb{R}^3$ into $\mathbb{R}^4$.
For example, Alexa et al. \cite{alexa2020properties} consider $\mathbb{R}^3\hookrightarrow\mathbb{S}^4$ by identifying $\mathbb{R}^3$ as a ``plane" in $\mathbb{R}^4$ and taking the inverse stereographic projection onto $\mathbb{S}^3$, and then they numerically check which discrete mean curvature approaches $1$ when the mean edge length converges to zero. Since the equation \eqref{second} in the smooth case holds for all hypersurfaces in $\mathbb{R}^4$, so in order to determine the primal or dual Laplacian that characterizes the optimal approximation to the Laplacian in $\mathbb{R}^3$, one must examine \eqref{discreteH} of all embeddings, which is not feasible. Therefore, it still lacks a rigorous proof that the dual Laplacian provides better description of the shape operator under immersion.

\iffalse
\textcolor{red}{\begin{ancillary}
For any triangular mesh $\mathcal{M}$ in $\mathbb{R}^3$, the discrete Laplacian derived in \cite{pinkall1993,polthier2002,gu2008} satisfies the equation \eqref{discreteH}, which is a discrete expression of \eqref{second}; that is, the discrete Laplacian acting on the coordinate function on $\mathcal{M}$ is equal to its mean curvature vector. However, for any fixed tetrahedral mesh $\mathcal{M}$, we can not directly check the equation \eqref{discreteH} for the primal and dual discrete Laplacian. We live in three dimensions and the coordinate functions depend on the embedding $\Phi: \mathcal{M}\hookrightarrow \mathbb{R}^4$, so in order to determine the primal or dual Laplacian that characterizes the optimal approximation to the Laplacian in $\mathbb{R}^3$, one must examine \eqref{discreteH} of all embeddings, which is not feasible. Even so, we can compare which structure gives the optimal approximation to the Laplacian in $\mathbb{R}^3$ if it involves a classical immersion from $\mathbb{R}^3$ into $\mathbb{R}^4$. 
\end{ancillary}}
\fi

%--------------------------------------------------------------------------
\subsection{Theoretical foundation}
\begin{figure}
\begin{subfigure}[h]{0.5\textwidth}
   \includegraphics[width=1.\linewidth]{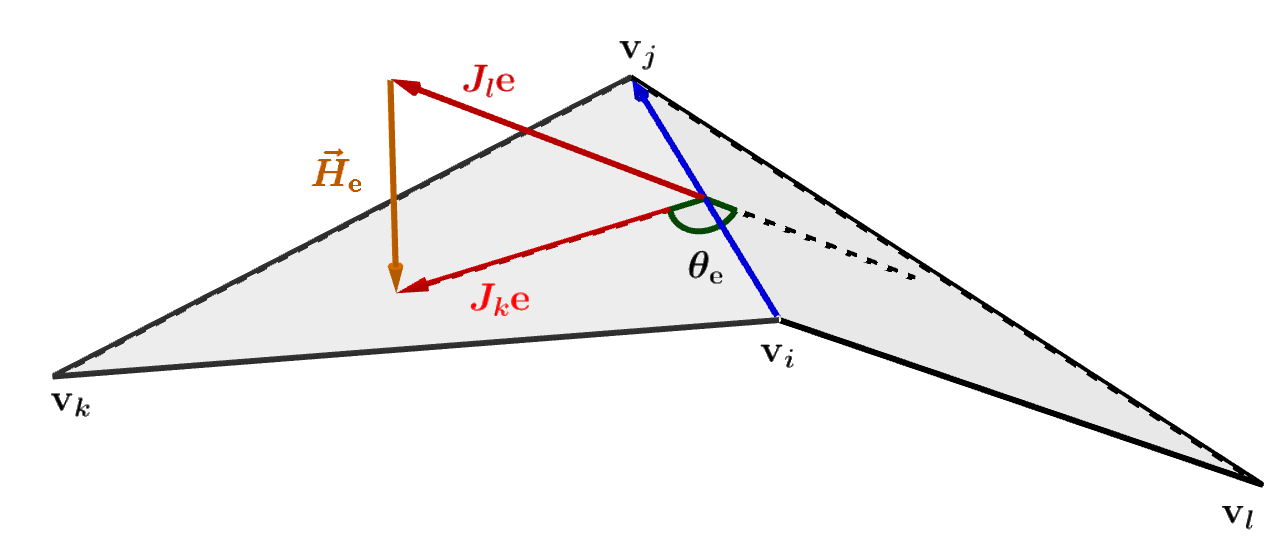}
   \caption{$\vec{H}_{\mathbf{e}}$}
   %\label{fig:pri_dual}
\end{subfigure}
\begin{subfigure}[h]{0.5\textwidth}
   \includegraphics[width=1.\linewidth]{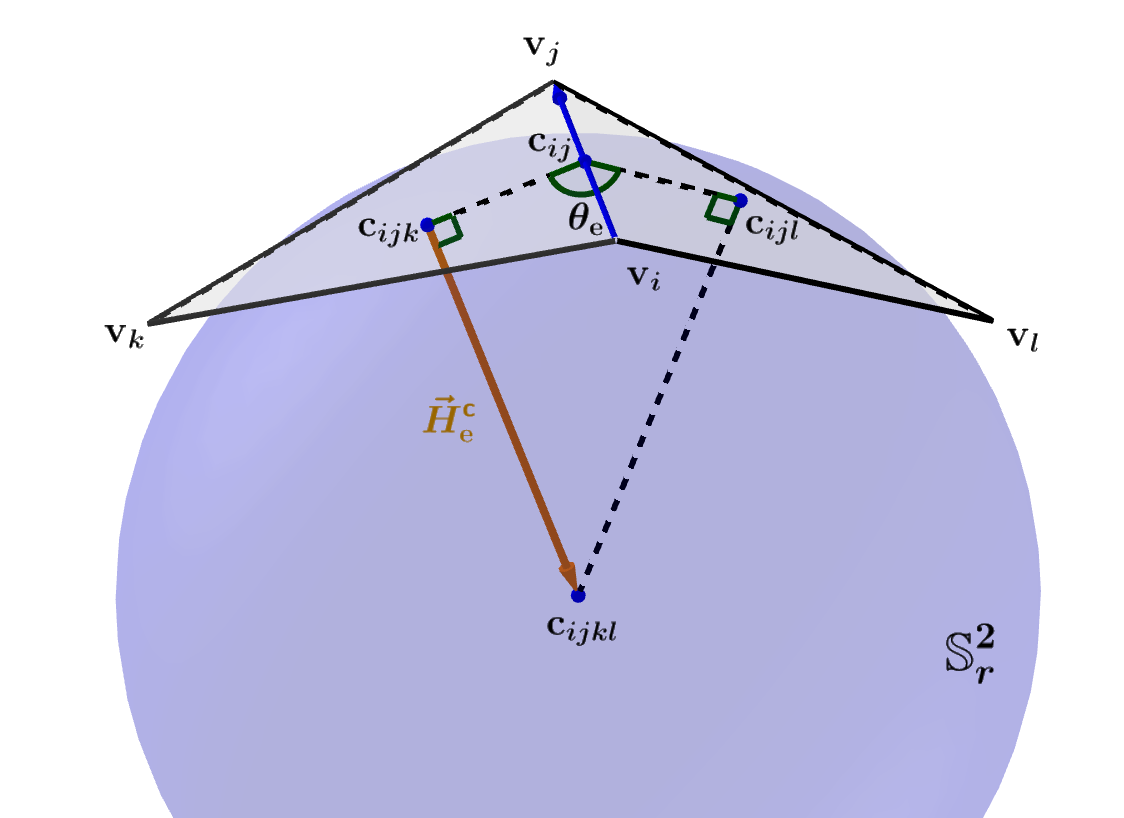}
   \caption{$\vec{H}_{\mathbf{e}}^{\mathsf{c}}$}
   %\label{fig:pri_hodge_dual}
\end{subfigure}
\caption{\label{fig:DiscreteMean}(a) The discrete mean curvature vector $\vec{H}_{\mathbf{e}}$ along the edge $\mathbf{e}:=[\mathbf{v}_i, \mathbf{v}_j]$. (b) The reciprocal discrete mean curvature vector $\vec{H}_{\mathbf{e}}^{\mathsf{c}}$ along the edge $\mathbf{e}:=[\mathbf{v}_i, \mathbf{v}_j]$.}
\end{figure}
\iffalse
\begin{figure*}[htb]
%  \centering
\centering
  \mbox{(a)}
  % the following command controls the width of the embedded PS file
  % (relative to the width of the current column)
  \includegraphics[width=.45\linewidth]{figures/DiscreteMean_edge.pdf}
  % replacing the above command with the one below will explicitly set
  % the bounding box of the PS figure to the rectangle (xl,yl),(xh,yh).
  % It will also prevent LaTeX from reading the PS file to determine
  % the bounding box (i.e., it will speed up the compilation process)
  % \includegraphics[width=.3\linewidth, bb=39 696 126 756]{sampleFig}
\centering
 \mbox{(b)}
  \includegraphics[width=.45\linewidth]{figures/DiscreteMean_edge1.pdf}
  \caption{\label{fig:DiscreteMean}%
           (a) The discrete mean curvature vector $\vec{H}_e$ along the edge $\mathbf{e}:=[\mathbf{v}_i, \mathbf{v}_j]$ is described by the vectors $J_k \mathbf{e}$ and $J_l \mathbf{e}$. (b) The reciprocal discrete mean curvature vector $\vec{H}_{\mathbf{e}}^{\mathsf{c}}$ is determined by the radius $r$ of the tangent sphere with tangent points $\mathbf{c}_{ijk}$ and $\mathbf{c}_{ijl}$.}
\end{figure*}
\fi
The primal Laplacian \eqref{primalL3} and the dual Laplacian \eqref{dualL3} are expressed in different forms; in other words, they construct different coordinate systems to describe the Laplacian in $\mathbb{R}^3$. Since the coordinate systems constructed by the primal and dual constructions are ambiguous in the ambient space $\mathbb{R}^3$, we consider the derivatives of functions and vector fields on hypersurfaces from a projective perspective.

Let $\mathcal{M}$ be a hypersurface in $\mathbb{R}^3$. For any point $x\in\mathcal{M}$ and vector $\mathbf{v}\in\mathbb{R}^3$, the projection from $\mathbb{R}^3$ to $T_x\mathcal{M}$ is defined as follows,
\begin{align*}
P_{T_x\mathcal{M}}(\mathbf{v})=\mathbf{v}-(\mathbf{v}\cdot \mu(x))\mu(x),\quad \mu(x)\in T_x^{\perp}\mathcal{M}.
\end{align*}
Let $U\supset\mathcal{M}$ be an open subset of $\mathbb{R}^3$. Assuming that $h:U\to\mathbb{R}$ is a differentiable function and $X:\mathcal{M}\to\mathbb{R}^3$ is a differentiable vector field, we define the tangential gradient of $h$ and the tangential divergence of $X$ with respect to $\mathcal{M}$ as
\begin{align*}
\nabla h(x)=&P_{T_x\mathcal{M}}(Dh(x))=Dh(x)-(Dh(x)\cdot\mu(x))\mu(x),\\
\mathrm{div}_{\mathcal{M}}X(x)=&\mathrm{div}_{\mathbb{R}^3}X(x)-D_{\mu(x)}X(x)\cdot\mu(x)
\end{align*}
where $D$ is the classical gradient operator in $\mathbb{R}^3$.
 
Take $X(x):= \nabla h(x)$, we have
\begin{align}\label{laplacianR3}
\Delta_{\mathcal{M}}h\equiv\mathrm{div}_{\mathcal{M}}\nabla h
=&\mathrm{div}_{\mathcal{M}}Dh-\mathrm{div}_{\mathcal{M}}\left((Dh\cdot\mu)\mu\right)=\mathrm{div}_{\mathcal{M}}Dh+\vec{H}\cdot Dh\nonumber\\
=&\Delta_{\mathbb{R}^3}h-D^2h(\mu, \mu)+\vec{H}\cdot Dh,
\end{align}
where $D^2h(\mu, \mu)$ is the second derivative of $h$ in the direction of $\mu\in T_x^{\perp}\mathcal{M}$ and $\vec{H}$ is the mean curvature vecture.

Consider a tetrahedral mesh $\mathcal{M}$ with each tetrahedron $\tau\in\mathcal{M}$ embedded by $\mathbf{f}$ in $\mathbb{R}^3$. Since the vertices of $\tau$ under the embedding lie on its circumsphere $\mathbb{S}_{\tau}$, the second derivative of $f:=\mathbf{f}\rvert_{\tau}$ in the normal direction $\mu\in T_{\mathbf{v}}^{\perp}\mathbb{S}_{\tau}$ vanishes for all vertices $\mathbf{v}\in \tau$. Note that,  according to Remark~\ref{rmk:normal}, the normal direction to each vertex of the tetrahedron is the radial direction of its circumsphere. Thus, the equation \eqref{laplacianR3} on the tetrahedron $\tau$ can be rewritten as
\begin{align}\label{laplacianR3_1}
\Delta_{\mathbb{R}^3}f=\Delta_{\partial\tau}f-\vec{H}\cdot Df.
\end{align}
In other words, we have the splitting formula for the Laplacian in $\mathbb{R}^3$ into the Laplacian of $\mathcal{M}$ and the mean curvature term associated with the embedding. Most importantly, with the equation \eqref{eq:EL}\iffalse and the arguments in \hyperref[AppendixA]{Appendix A}\fi, we can solve the critical point of the Dirichlet energy by requiring that the left-hand side of \eqref{laplacianR3_1} be equal to zero.

%--------------------------------------------------------------------------
\subsection{Associative Laplacian for the dual construction}
Before we prove that the dual Laplacian given by \eqref{dualL}, \eqref{dualL1} and \eqref{dualL2_3} provides the optimal approximation to the smooth Laplacian of $\mathbb{R}^3$, we consider the associated Laplacian according to the dual Laplacian. First, we give an ansatz $\widetilde{w}_{ij}$ and $\widetilde{w}_{ijkl}$ associated with $w_{ij}$ and $w_{ijkl}$ in \eqref{dualL2_3} as follows,
\begin{subequations}\label{ansatz}
\begin{align}
\widetilde{w}_{ij}=&\sum_{[\mathbf{v}_i, \mathbf{v}_j, \mathbf{v}_k, \mathbf{v}_l]}\widetilde{w}_{ijkl}+\widetilde{w}_{ijlk},\\
\widetilde{w}_{ijkl}
=&\frac{|[\mathbf{v}_i, \mathbf{v}_j]|}{8}\left(\frac{\cot\alpha_{ij}+\cot\beta_{ij}}{2}\right)^{\top}\nonumber\\
&+\frac{|[\mathbf{v}_i, \mathbf{v}_j]|}{8}\left(\frac{1-\cos\theta_{ij}^{kl}}{\sin\theta_{ij}^{kl}}\right)^2
\left[\left(\frac{\cot\alpha_{ij}+\cot\beta_{ij}}{2}\right)^{\perp}\right]^{3}.
\end{align}
\end{subequations}

Without loss of generality, let $\mathcal{S}$ be a closed, simply--connected two--dimensional triangular surface. Given a vertex $\mathbf{v}_i\in\mathcal{S}$ and any vertex $\mathbf{v}_j$ adjacent to $\mathbf{v}_i$, we denote $\alpha_{ij}$ and $\beta_{ij}$ be the angles opposite the edge $\mathbf{e}:=[\mathbf{v}_i, \mathbf{v}_j]$ in the two incident triangles and $\theta_{\mathbf{e}}$ be the dihedral angle on the edge $[\mathbf{v}_i, \mathbf{v}_j]$. Based on the weight $\widetilde{w}_{ijkl}$ in \eqref{ansatz}, we introduce the reciprocal discrete mean curvature corresponding to the edge $\mathbf{e}$.

\begin{definition}\label{discreteHcd}
Suppose $\mathrm{T}_k:= [\mathbf{v}_i, \mathbf{v}_j, \mathbf{v}_k]$ and $\mathrm{T}_l:= [\mathbf{v}_i, \mathbf{v}_j, \mathbf{v}_l]$ are two non--coplanar triangles with a common edge $\mathbf{e}:=[\mathbf{v}_i, \mathbf{v}_j]$ and $\theta_{\mathbf{e}}$ is the dihedral angle at $\mathbf{e}$ enclosed by $\mathrm{T}_k$ and $\mathrm{T}_l$.  
The operator $J_k$ $(J_l)$ is a complex structure in the plane of the triangle $\mathrm{T}_k$ $(\mathrm{T}_l)$, by rotating $90^{\circ}$, and then the reciprocal discrete mean curvature vector to the edge $\mathbf{e}$ is defined by
\begin{align}\label{discreteHc}
2\vec{H}_{\mathbf{e}}^{\mathsf{c}}:=\left(\frac{1-\cos\theta_{\mathbf{e}}}{\sin\theta_{\mathbf{e}}}\right)J\mathbf{e},\quad J\mathbf{e}:=(J_{l,k}\circ J_{k})\mathbf{e}
\end{align}
where $J_{l,k}$ is the complex structure on the plane spanned by $\{J_l\mathbf{e}, J_k\mathbf{e}\}$ with positive orientation. Obviously, $J\mathbf{e}$ is perpendicular to $J_k\mathbf{e}$. The reciprocal discrete mean curvature is given by
\begin{align}\label{discreteHcq}
|\vec{H}_{\mathbf{e}}^{\mathsf{c}}|=\frac{1-\cos\theta_{\mathbf{e}}}{\sin\theta_{\mathbf{e}}}\frac{|\mathbf{e}|}{2}.
\end{align}
Specifically, $\theta_{\mathbf{e}}=0$ means that the tetrahedron with the faces $\mathrm{T}_k$ and $\mathrm{T}_l$ degenerates to a triangle and the reciprocal discrete mean curvature becomes zero. Furthermore, $\theta_{\mathbf{e}}=\pi$, the reciprocal discrete mean curvature becomes infinity.
\end{definition}

\begin{remark}
We will give some explanations for the words used in defining the ``associated" Laplacian in \eqref{ansatz} and the ``reciprocal" discrete mean curvature in \eqref{discreteHc}.
\begin{itemize}
\item``Associated":\\
In order to show that the dual Laplacian in \eqref{dualL3} satisfies the geometric property \eqref{laplacianR3_1}, our strategy is to determine the upper estimate of the discrete Dirichlet energy $-\mathbf{f}^{\intercal}\mathbf{L}\mathbf{f}$ for  $\mathbf{L}=[w_{ij}]$ in \eqref{dualL3}.
Based on the estimates in \hyperref[AppendixB]{Appendix A}, we define the associated Laplacian satisfying the geometric property \eqref{laplacianR3_1} in the weak* sense and its discrete Dirichlet also gives the upper estimate of the discrete Dirichlet energy for dual Laplacian. As long as the mesh size tends to 0, the discrete Dirichlet of the associated Laplacian forces the discrete Dirichlet energy of the dual Laplacian to converge to the smooth Dirichlet energy, and so we can conclude that the dual Laplacian must satisfy \eqref{laplacianR3_1} in a weak* sense.
\item ``Reciprocal":\\
In light of \textup{Figure ~\ref{fig:DiscreteMean} (b)} and the discussion in Section~\ref{Geoperps}, we will realize that the tangent sphere at the points $\mathbf{c}_{ijk}$ and $\mathbf{c}_{ijl}$ provides an internal approximation for discrete mean curvature along the edges. Besides, the quantity $|\vec{H}_{\mathbf{e}}^{\mathsf{c}}|$ in \eqref{discreteHcq} is the radius of the tangent sphere and its reciprocal $|\vec{H}_{\mathbf{e}}^{\mathsf{c}}|^{-1}$ is the curvature of the tangent sphere.
\end{itemize}
\end{remark}

%--------------------------------------------------------------------------
\subsection{Geometric perspective of $\vec{H}_{\mathbf{e}}^{\mathsf{c}}$}\label{Geoperps}
Suppose $\mathrm{T}_k:= [\mathbf{v}_i, \mathbf{v}_j, \mathbf{v}_k]$ and $\mathrm{T}_l:= [\mathbf{v}_i, \mathbf{v}_j, \mathbf{v}_l]$ are two non--coplanar triangles with common edge $\mathbf{e}:=[\mathbf{v}_i, \mathbf{v}_j]$ and $\theta_{\mathbf{e}}$ is the dihedral angle of the edge $\mathbf{e}$. In the planes of the triangle $\mathrm{T}_k$ and $\mathrm{T}_l$, we denote $J_k$ and $J_l$ as the complex structures rotated by $90^{\circ}$, respectively. After that, $J$ is the complex structure associated with the plane $\mathrm{Span}\{J_k\mathbf{e}, J_l\mathbf{e}\}$ constructed in Definition~\ref{discreteHcd}.

In the following arguments, we can see Figure~\ref{fig:DiscreteMean} (a) for description.
According to Definition~\ref{discreteEdge}, the discrete mean curvature along the edge $\mathbf{e}$ is given by
\begin{align*}
\vec{H}_{\mathbf{e}}=J_k\left(\frac{\mathbf{e}}{2}\right)-J_l\left(\frac{\mathbf{e}}{2}\right),\quad |\vec{H}_{\mathbf{e}}|=\sin((\pi-\theta_{\mathbf{e}})/2)|\mathbf{e}|.
\end{align*}
Furthermore, $2|\vec{H}_{\mathbf{e}}|/|\mathbf{e}|=2\sin((\pi-\theta_{\mathbf{e}})/2)$ gives an approximation to the sweep angle between the unit vectors $J_k\left(\frac{\mathbf{e}}{2}\right)$ and $J_l\left(\frac{\mathbf{e}}{2}\right)$ across the edge $\mathbf{e}$.

Similarly, we study the reciprocal discrete mean curvature along the edge $\mathbf{e}$ in Definition ~\ref{discreteHcd}.
\iffalse
as follows,
\begin{align*}
2\vec{H}_e^{\mathsf{c}}\equiv \left(\frac{1-\cos\theta_e}{\sin\theta_e}\right)Je,\quad |\vec{H}_e^{\mathsf{c}}|=\frac{1-\cos\theta_e}{\sin\theta_e}\frac{|e|}{2}.
\end{align*}
\fi
Use Figure~\ref{fig:DiscreteMean} (b) as a reference, we consider the quadriateral  in the plane $*\mathbf{e}$ spanned by $\mathbf{c}_{ij}$, $\mathbf{c}_{ijk}$, $\mathbf{c}_{ijkl}$ and $\mathbf{c}_{ijl}$, which are defined as follows 
\begin{subequations}\label{quadpts}
\begin{align}
&\mathbf{c}_{ij}=\frac{\mathbf{v}_i+\mathbf{v}_j}{2}, \quad |\mathbf{c}_{ij}-\mathrm{v}_{i}|=|\mathbf{e}|/2,\\
&\mathbf{c}_{ijk}=\mathbf{c}_{ij}+\frac{1}{2}J_k(\mathbf{v}_j-\mathbf{v}_i), \quad |\mathbf{c}_{ijk}-\mathbf{c}_{ij}|=|\mathbf{e}|/2,\\
&\mathbf{c}_{ijl}=\mathbf{c}_{ij}-\frac{1}{2}J_l(\mathbf{v}_j-\mathbf{v}_i), \quad |\mathbf{c}_{ijl}-\mathbf{c}_{ij}|=|\mathbf{e}|/2.
\end{align}
Specifically, the point $\mathbf{c}_{ijkl}$ is determined by
\begin{align}
|\mathbf{c}_{ijkl}-\mathbf{c}_{ijk}|=&|\mathbf{c}_{ijkl}-\mathbf{c}_{ijl}|=h,\\
(\mathbf{c}_{ijk}-\mathbf{c}_{ij})\cdot(\mathbf{c}_{ijk}-\mathbf{c}_{ijkl})=&(\mathbf{c}_{ijl}-\mathbf{c}_{ij})\cdot(\mathbf{c}_{ijkl}-\mathbf{c}_{ijl})=0.\label{eq:inner}
\end{align}
\end{subequations}
Note that $\mathbf{c}_{ijk}$, $\mathbf{c}_{ijkl}$ and $\mathbf{c}_{ijl}$ do not necessarily lie in the tetrahedron. 

Choosing the plane $*\mathbf{e}:=\mathbb{R}^2$, we can select the coordinate system shown in Figure~\ref{fig:refercoord}: translate $\mathbf{c}_{ij}$ to the origin and place the edge $[\mathbf{c}_{ij}, \mathbf{c}_{ijk}]$ on the positive $x-$axis and then re-coordinate  $\mathbf{c}_{ij}$, $\mathbf{c}_{ijk}$, $\mathbf{c}_{ijkl}$ and $\mathbf{c}_{ijl}$ given in \eqref{quadpts} as follows
\begin{align}\label{recoord}
\mathbf{c}_{ij}=(0, 0),\quad \mathbf{c}_{ijk}=(|\mathbf{e}|/2, 0),\quad
\mathbf{c}_{ijl}=|\mathbf{e}|/2(\cos\theta_{\mathbf{e}}, \sin\theta_{\mathbf{e}}),\quad \mathbf{c}_{ijkl}=(|\mathbf{e}|/2, h).
\end{align}
Applying the inner product in \eqref{eq:inner} to the given coordinates \eqref{recoord} , we obtain an expression for $h$:
\begin{align}
h=\frac{|\mathbf{e}|}{2}\frac{1-\cos\theta_{\mathbf{e}}}{\sin\theta_{\mathbf{e}}}.
\end{align} 
From the geometric perspective in the cross-section $*\mathbf{e}$, the circle centered at $\mathbf{c}_{ijkl}$ with radius $h$ is tangent to both planes containing $\mathrm{T}_k$ and $\mathrm{T}_l$ at points $\mathbf{c}_{ijk}$ and $\mathbf{c}_{ijl}$, respectively, so the curvature of this circle is the reciprocal of its radius. Therefore, the reciprocal discrete mean curvature gives an approximation of bending on the edge $\mathbf{e}$ by $1/|\vec{H}_{\mathbf{e}}^{\mathsf{c}}|$. 

\begin{figure}[t]
  \centering
  % the following command controls the width of the embedded PS file
  % (relative to the width of the current column)
  \includegraphics[width=0.7\linewidth]{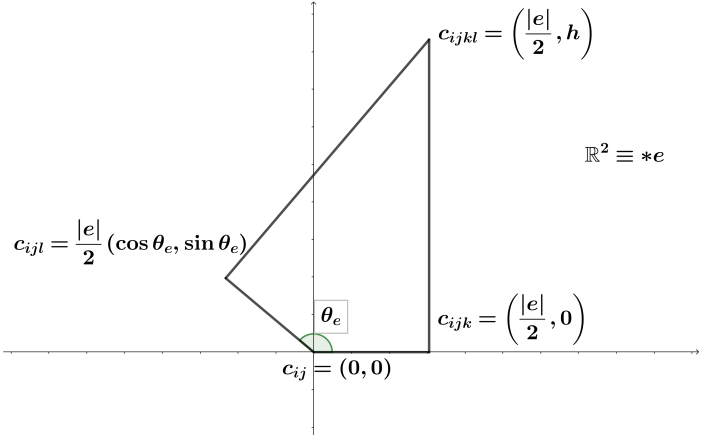}
  % replacing the above command with the one below will explicitly set
  % the bounding box of the PS figure to the rectangle (xl,yl),(xh,yh).
  % It will also prevent LaTeX from reading the PS file to determine
  % the bounding box (i.e., it will speed up the compilation process)
  % \includegraphics[width=.95\linewidth, bb=39 696 126 756]{sampleFig}
  %
  %\parbox[t]{.9\columnwidth}{\relax
     %      For all figures please keep in mind that you \textbf{must not}
        %   use images with transparent background! 
           %}
  %
  \caption{\label{fig:refercoord}
           Select the coordinates on the plane $*\mathbf{e}$ and redefine the coordinates of $\mathbf{c}_{ij}$, $\mathbf{c}_{ijk}$, $\mathbf{c}_{ijkl}$ and $\mathbf{c}_{ijl}$.}
\end{figure}

\begin{remark}
For the smooth surfaces, the sectional curvature of the surface is determined by the curvature of the curve where the selected section intersects the surface. For discrete surfaces, Definition~\ref{discreteEdge} and Definition~\ref{discreteHcd} are based on the idea of defining the discrete mean curvature on an edge $\mathbf{e}$ over the section spanned by $\{J_k\mathbf{e}, J_l\mathbf{e}\}$. Comparing these methods, Definition~\ref{discreteHcd} provides a quadratic approximation to the mean curvature on the edge $\mathbf{e}$. The numerical accuracy of these methods is shown in Table.
\end{remark}

%--------------------------------------------------------------------------
\section{Main theorem}\label{sec:6}
Let $\mathcal{M}_{\delta}$ be a triangle mesh with mean edge length $\delta$. For any two adjacent triangles in $\mathcal{M}_{\delta}$, we can construct a tetrahedron by connecting the opposite vertices with a line segment to obtain a tetrahedral mesh $\widetilde{\mathcal{M}}_{\delta}$. Note that the tetrahedra in $\widetilde{\mathcal{M}}_{\delta}$ may overlap, but the data used to construct the dual Laplacian restricted to $\mathcal{M}_{\delta}$ are not affected by these overlaps. More specifically, we can define the cotangent Laplacian $\mathbf{L}$ for a triangle mesh $\mathcal{M}_{\delta}$ with weights $w_{ij}=\frac{1}{2}\left(\cot\alpha_{ij}+\cot\beta_{ij}\right)$ in \eqref{triangle}.

Replacing $w_{ij}$ in $\mathbf{L}$ by weights $ w^d_{ij}$ in \eqref{dualL2_3}
\begin{align*}
\begin{split}
w^d_{ij}=&\sum_{[\mathbf{v}_i, \mathbf{v}_j, \mathbf{v}_k, \mathbf{v}_l]}w_{ijkl}+w_{ijlk}\\
w^d_{ijkl}=&\frac{|[\mathbf{v}_i, \mathbf{v}_j]|}{8}\cot\theta_{ij}^{kl}\left(\frac{2\cot\alpha_{ij}\cot\beta_{ij}}{\cos\theta_{ij}^{kl}}-(\cot^2\alpha_{ij}+\cot^2\beta_{ij})\right).
\end{split}
\end{align*}
yields the dual Laplacian $\mathbf{L}^d$ for the tetrahedral mesh $\widetilde{\mathcal{M}}_{\delta}$ restricted to $\mathcal{M}_{\delta}$.
In the restriction of $\mathcal{M}_{\delta}$, the associated Laplacian $\widetilde{\mathbf{L}}$ can be further deduced from the dual Laplacian $\mathbf{L}^d$ by replacing $w^d_{ij}$ by the weights $\widetilde{w}_{ijkl}$ in \eqref{ansatz}.
\begin{align*}
\begin{split}
\widetilde{w}_{ij}=&\sum_{[\mathbf{v}_i, \mathbf{v}_j, \mathbf{v}_k, \mathbf{v}_l]}\widetilde{w}_{ijkl}+\widetilde{w}_{ijlk}\\
\widetilde{w}_{ijkl}
=&\frac{|[\mathbf{v}_i, \mathbf{v}_j]|}{8}\left(\frac{\cot\alpha_{ij}+\cot\beta_{ij}}{2}\right)^{\top}\nonumber\\
&+\frac{|[\mathbf{v}_i, \mathbf{v}_j]|}{8}\left(\frac{1-\cos\theta_{ij}^{kl}}{\sin\theta_{ij}^{kl}}\right)^2
\left[\left(\frac{\cot\alpha_{ij}+\cot\beta_{ij}}{2}\right)^{\perp}\right]^{3}.
\end{split}
\end{align*}

\begin{theorem}\label{mainthm}
Assuming that $\mathcal{M}_{\delta}$ and $\widetilde{\mathcal{M}}_{\delta}$ are the triangle and the tetrahedral mesh, the discrete Laplacian matrix of $\mathcal{M}_{\delta}$ and the associated Laplacian matrix of $\widetilde{\mathcal{M}}_{\delta}$ are given by
$
\mathbf{L}=(w_{ij})$ and $\widetilde{\mathbf{L}}=(\widetilde{w}_{ij})
$
where $w_{ij}$ is determined by \eqref{primalL2} and \eqref{primalL3} and $\widetilde{w}_{ij}$ is determined by \eqref{dualL2} and \eqref{ansatz}.
Then for all piecewise linear functions $\mathbf{f}$, we have the weak* convergence  
\begin{align}
\widetilde{\mathbf{L}}_{ij}\stackrel{\ast}{\rightharpoonup}\mathbf{L}_{ij}-\vec{H}_{[\mathbf{v}_i, \mathbf{v}_j]}
\end{align}
when $\delta\to 0$, and it is unique up to a constant.
\end{theorem}
\begin{remark}
Based on the approach \eqref{approach} i.e., $\left(-\Delta f, f\right)_{L^2} \approx -\mathbf{f}^{\intercal}\mathbf{L}\mathbf{f}$, the cotangent weight Laplacian is implicitly defined by the energy functional (Dirichlet energy). Proving the convergence of the cotangent weight Laplacian is equivalent to demonstrating the convergence of the energy functional. Unlike the convergence of functions, we want to prove the convergence of functionals in (algebraic) dual space and therefore must consider weak* convergence.
\end{remark}
 
\begin{proof}
Let %$\mathrm{X}_{\delta}$ be a set defined by
$
\mathrm{X}_{\delta}=\left\{\mathbf{e}_{ij}:=\frac{[\mathbf{v}_i, \mathbf{v}_j]}{|[\mathbf{v}_i, \mathbf{v}_j]|}\bigg|  [\mathbf{v}_i, \mathbf{v}_j]\in\mathcal{M}_{\delta}\right\}.
$ Then $\mathrm{X}\equiv\bigcup_{\delta>0}\mathrm{X}_{\delta}$ is a family of linearised characteristic functions for points on a smooth surface. Since we assume that the unknown function $\mathbf{f}$ is a piecewise linear approximation of the unknown exact solution, or that we consider a gradient approximation based on linearization along the edges i.e., $(\mathbf{v}_j-\mathbf{v}_i)\nabla f=(f_j-f_i)$ for $[\mathbf{v}_i, \mathbf{v}_j]\in\mathcal{M}_{\delta}$, so we only need to prove for $\mathbf{e}_{ij}\in\mathrm{X}$ and $\delta\to 0$,
$
\widetilde{\mathbf{L}}_{ij}\stackrel{\ast}{\rightharpoonup}\mathbf{L}_{ij}-\vec{H}_{[\mathbf{v}_i, \mathbf{v}_j]}.
$ 

Acting on the edges $[\mathbf{v}_i,\mathbf{v}_j]$ with weights $\widetilde{w}_{ijkl}$, the first term in \eqref{ansatz} is in the expression $\frac{1}{8}\mathbf{L}_{ij}$, so we can prove that for $\mathbf{e}_{ij}\in\mathrm{X}$ and $\delta\to 0$
\begin{align*}
\frac{|[\mathbf{v}_i, \mathbf{v}_j]|}{8}\left(\frac{1-\cos\theta_{ij}^{kl}}{\sin\theta_{ij}^{kl}}\right)^2
\left[\left(\frac{\cot\alpha_{ij}+\cot\beta_{ij}}{2}\right)^{\perp}\right]^{3}\stackrel{\ast}{\rightharpoonup}\frac{-1}{8}\vec{H}_{[\mathbf{v}_i, \mathbf{v}_j]}.
\end{align*}
Based on the equation in \eqref{discreteH}, we simplify the argument and prove that for $\mathbf{e}_{ij}\in\mathrm{X}$ and $\delta\to 0$
\begin{align}\label{wstar}
\left(\frac{1-\cos\theta_{ij}^{kl}}{\sin\theta_{ij}^{kl}}\right)^2
\left[\left(\frac{\cot\alpha_{ij}+\cot\beta_{ij}}{2}\right)^{\perp}\right]^{2}\stackrel{\ast}{\rightharpoonup} 1.
\end{align}
Considering the discrete mean curvature on $[\mathbf{v}_i, \mathbf{v}_j]$, we use the equivalent arguments of the discrete mean curvature at a vertex and on an edge in Remark~\ref{rmkVE}, which shows that
\begin{align*}\label{normHedge1}
\left|\left(\frac{\cot\alpha_{ij}+\cot\beta_{ij}}{2}\right)^{\perp}\mathbf{e}_{ij}\right|^2=\left|\mathbf{L}^{\perp}_{ij}\frac{[\mathbf{v}_i, \mathbf{v}_j]}{|[\mathbf{v}_i, \mathbf{v}_j]|}\right|^2=\frac{\left|\vec{H}_{[\mathbf{v}_i, \mathbf{v}_j]}\right|^2}{|[\mathbf{v}_i, \mathbf{v}_j]|^2}=\sin^2((\pi-\theta_{ij}^{kl})/2).
\end{align*}

For any $\mathbf{e}_{ij}\in\mathrm{X}$ and any point $\mathbf{v}$ on a smooth surface , there exists a sequence of triangles $\{[\mathbf{v}^s_i, \mathbf{v}^s_j, \mathbf{v}^s_k], [\mathbf{v}^s_i, \mathbf{v}^s_j, \mathbf{v}^s_l]\}_{s=1}^{\infty}$ with their mean edge length $\delta_s$ satisfying $\mathbf{v}^s_i=\mathbf{v}$
\begin{align}
[\mathbf{v}^s_i, \mathbf{v}^s_j]=C_s \mathbf{e}_{ij}, \quad C_s:= |[\mathbf{v}^s_i, \mathbf{v}^s_j]|.
\end{align}

Based on the convergent argument in \cite{cheeger1984}, the triangles $[\mathbf{v}^s_i, \mathbf{v}^s_j, \mathbf{v}^s_k]$ and $[\mathbf{v}^s_i, \mathbf{v}^s_j, \mathbf{v}^s_l]$ converge to triangles $[\mathbf{v}, \mathbf{v}^*_j, \mathbf{v}^*_k]$ and $[\mathbf{v}, \mathbf{v}^*_j, \mathbf{v}^*_l]$ on the tangent plane at point $\mathbf{v}$, where the dihedral angle $\theta_{ij}^{kl}$ on the edge $[\mathbf{v}^s_i, \mathbf{v}^s_j]$  becomes to $\pi$ as $\delta_s\to 0$. Therefore, we have 
\begin{align*}
&\lim_{\delta_s\to0}\sup_{\substack{\mathbf{e}_{ij}\in\mathrm{X}\\|\mathbf{e}_{ij}|=1}}\left\{\left|\left(\frac{1-\cos\theta_{ij}^{kl}}{\sin\theta_{ij}^{kl}}\right)\left(\frac{\cot\alpha_{ij}+\cot\beta_{ij}}{2}\right)^{\perp}\mathbf{e}_{ij}\right|^2\right\}\\
=&\lim_{\delta_s\to0}\left\|\left(\frac{1-\cos\theta_{ij}^{kl}}{\sin\theta_{ij}^{kl}}\right)\vec{H}_{\mathbf{e}_{ij}}\right\|^2
=\lim_{\delta_s\to0}\left\|\left(\frac{1-\cos\theta_{ij}^{kl}}{\sin\theta_{ij}^{kl}}\right)\sin((\pi-\theta_{ij}^{kl})/2)\right\|^2=1
\end{align*}
That is, as the operator, $\left(\frac{1-\cos\theta_{ij}^{kl}}{\sin\theta_{ij}^{kl}}\right)^2
\left[\left(\frac{\cot\alpha_{ij}+\cot\beta_{ij}}{2}\right)^{\perp}\right]^{2}\stackrel{\ast}{\rightharpoonup} 1$. Then we complete the proof.
\end{proof}

\begin{corollary}
Provided the mean edge length $\delta$ is sufficiently small, the dual construction in \eqref{dualL3} gives ``optimal" Laplacians in $\mathbb{R}^3$ compared to the primal construction in \eqref{primalL3}.
\end{corollary}
\begin{proof}
According to the argument in \hyperref[AppendixB]{Appendix A}, the associated Laplacian gives the bounded control of the dual construction. When the mean length $\delta$ is sufficiently small, the associated Laplacian satisfying the Euler--Lagrange equation and the dual Laplacian all have weak* converge to the continuous Laplacian. 
\end{proof}
After completing Theorem~\ref{mainthm}, we can define the associated mean curvature $|\vec{H}_a|$ on the edges $[\mathbf{v}_i, \mathbf{v}_j]$ according to  the conventions at the end of Section~\ref{sec:End4} by the equation below,
\begin{align}\label{assoMean}
|\vec{H}_a|
:=\left(\frac{1-\cos\theta_{ij}^{kl}}{\sin\theta_{ij}^{kl}}\right)^2\left[\left(\frac{\cot\alpha_{ij}+\cot\beta_{ij}}{2}\right)^{\perp}\right]^{3}=\left(\frac{1-\cos\theta_{ij}^{kl}}{\sin\theta_{ij}^{kl}}\right)^2\sin^3\left(\frac{\pi-\theta_{ij}^{kl}}{2}\right).
\end{align}

\section{Higher Order Approximation}\label{sec:7}
\begin{figure}
\begin{subfigure}[h]{0.45\textwidth}
   \includegraphics[width=1.\linewidth]{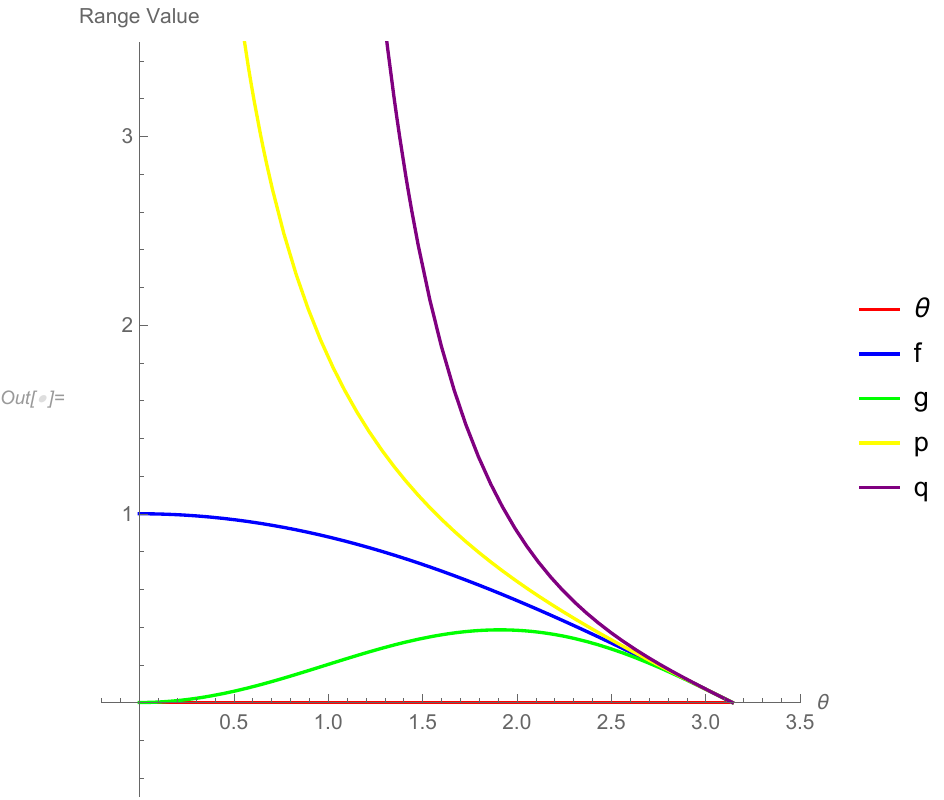}
   \caption{Discrete mean curvature candidates}
   \label{fig:DiscreteMeanCan}
\end{subfigure}
\begin{subfigure}[h]{0.5\textwidth}
   \includegraphics[width=1.1\linewidth]{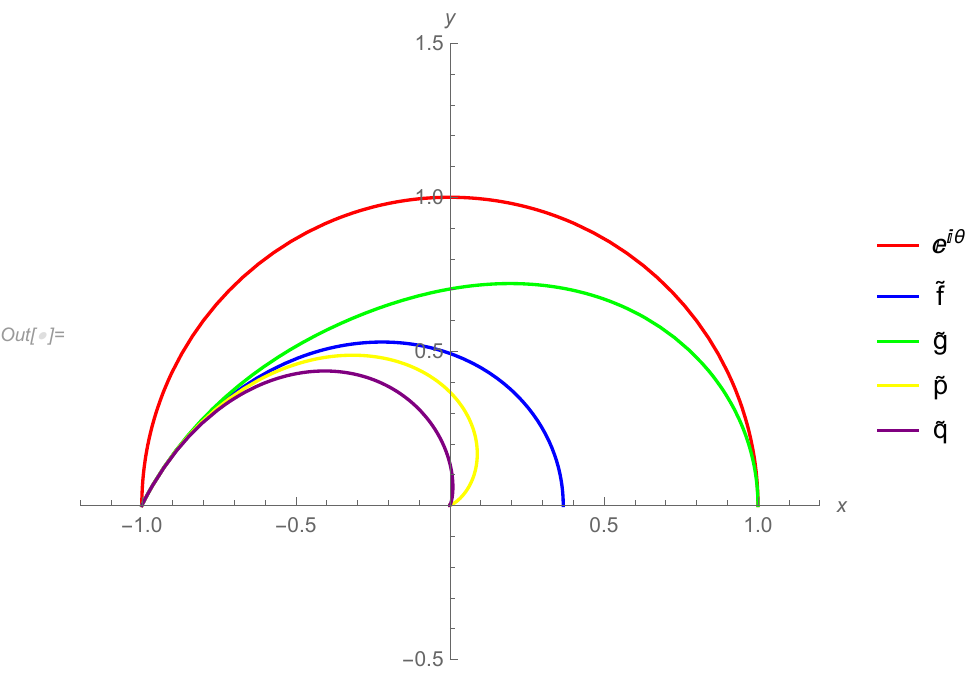}
   \caption{Mobious transformation of Discrete mean curvature candidates}
   \label{fig:DiscreteMeanCan1}
\end{subfigure}
\caption{\label{fig:DiscreteMeanCan}(a) Discrete mean curvatures as angle varies from $0$ to $\pi$ (b) The reciprocal discrete mean curvature vector $\vec{H}_{\mathbf{e}}^{\mathsf{c}}$ along the edge $\mathbf{e}:=[\mathbf{v}_i, \mathbf{v}_j]$.}
\end{figure}
The mean curvature of a point on a smooth manifold is defined as the average of the two principal curvatures. Thus, the same idea can be applied to points on the edges $[\mathbf{v}_i, \mathbf{v}_j]$ of the triangle mesh $\mathcal{M}$. A principal curvature along $[\mathbf{v}_i, \mathbf{v}_j]$ is clearly zero, since it is a straight line. Thus, the discrete mean curvature on $[\mathbf{v}_i, \mathbf{v}_j]$ is determined by the instantaneous rate of change of direction at the point that crosses the edge $[\mathbf{v}_i, \mathbf{v}_j]$ in the normal direction. In Figure~\ref{fig:DiscreteMeanCan1}, we use the fact \eqref{wstar} in the proof of Theorem~\ref{mainthm} to study the possible discrete mean curvatures associated with the cotangent--weighted Laplacian on the edges $[\mathbf{v}_i, \mathbf{v_j}]$, defined as follows:
\begin{itemize}
\item $f(\theta):=\sin\left(\frac{\pi-\theta}{2}\right)=\cos\frac{\theta}{2}$ \cite{polthier2002,sullivan2008}
\item $g(\theta):=\left(\frac{1-\cos\theta}{\sin\theta}\right)^2\sin^3\left(\frac{\pi-\theta}{2}\right)=\frac{1}{2}\sin\frac{\theta}{2}\sin\theta$ \eqref{assoMean}
\item $p(\theta):=\frac{\sin\theta}{1-\cos\theta}=\cot\frac{\theta}{2}$
\item $q(\theta):=\sin^{-2}\left(\frac{\pi-\theta}{2}\right) \left(\frac{\sin\theta}{1-\cos\theta}\right)^3=\frac{1}{2}\csc^4\frac{\theta}{2}\sin\theta$
\end{itemize}
In Figure~\ref{fig:DiscreteMeanCan}, they all describe local minimal surfaces when the dihedral angle on the edge tends to $\pi$, which means that for infinitesimal mesh size, they all agree with the classical mean curvature on smooth manifolds. However, mesh size cannot be infinitely small in practical applications, so we need to check their sensitivity to instantaneous changes in angle. We apply a M\"{o}bius transformation $e^{\mathrm{i}z}$ on the graphs of the above candidate functions and then we obtain the corresponding graphs $\widetilde{f}$, $\widetilde{g}$, $\widetilde{p}$, and $\widetilde{q}$ in a polar coordinate system as shown in Figures~\ref{fig:DiscreteMeanCan1}. 

Let $\Gamma(r, \theta)$ be a graph in $\mathbb{R}^2\cong\mathbb{C}$, we consider the Taylor expansion $\mathfrak{T}$ of the sum of the $x$-coordinate function ($\mathfrak{Re}(\Gamma(r, \theta))$) and the $y$-coordinate function ($\mathfrak{Im}(\Gamma(r, \theta))$) at $\theta=0$. The sensitivity to instantaneous changes in angle is a comparison of a Taylor expansion about $\Gamma(r,\theta)$ with a Taylor expansion about $e^{\mathrm{i}\theta}$ and the identical polynomial of the same degree represent the approximate order of the angle.
\begin{table}[b]
\centering
\caption{Function describing angular transients, the M\"{o}bius transformation on graph function, Taylor expansion of sum of coordinate functions, and the initial value}
\label{tab:Taylor}
    \begin{tabular}{|c|l|l|l|}
        \hline
\thead{Function} & \thead{M\"{o}bius transformation $e^{\mathrm{i}z}$}& \thead{The sum of Taylor expansions $\mathfrak{T}(\mathfrak{Re})$ and $\mathfrak{T}(\mathfrak{Im})$} & \thead{$\theta=0$}\\
        \hline
$f(\theta)$ & $\widetilde{f}:=e^{-\cos\frac{\theta}{2}}e^{\mathrm{i}\theta}$ & $\mathfrak{T}(\widetilde{f})=\frac{1}{e}+\frac{\theta}{e}-\frac{3\theta^2}{8e}-\frac{\theta^3}{24e}+\mathcal{O}(\theta^4)$ & $1$\\
        \hline
$g(\theta)$ & $\widetilde{g}:=e^{-\frac{1}{2}\sin\frac{\theta}{2}\sin\theta}e^{\mathrm{i}\theta}$ & $\mathfrak{T}(\widetilde{g})=1+\theta-\frac{3\theta^2}{4}-\frac{5\theta^3}{12}+O(\theta^4)$ & $0$\\
        \hline
$p(\theta)$ & $\widetilde{p}:=e^{-\cot\frac{\theta}{2}}e^{\mathrm{i}\theta}$ & $ \mathfrak{T}(\widetilde{p})=e^{-\frac{2}{\theta}+\mathcal{O}(\theta)}\left(1+\theta-\frac{\theta^2}{2}-\frac{\theta^3}{6}+\mathcal{O}(\theta^4)\right)$ & $\infty$\\
        \hline
$q(\theta)$ & $\widetilde{q}:=e^{-\frac{1}{2}\csc^4\frac{\theta}{2}\sin\theta}e^{\mathrm{i}\theta}$ & $\mathfrak{T}(\widetilde{q})=e^{-\frac{8}{\theta^3}+\mathcal{O}(\theta)}\left(1+\theta-\frac{\theta^2}{2}-\frac{\theta^3}{6}+\mathcal{O}(\theta^4)\right)$ & $\infty$\\
        \hline
$\theta$ & $\widetilde{\theta}:=e^{\mathrm{i}\theta}$ & $\mathfrak{T}(\widetilde{\theta})=1+\theta-\frac{\theta^2}{2}-\frac{\theta^3}{6}+\mathcal{O}(\theta^4)$ & $0$\\
        \hline
    \end{tabular}
\end{table}
\iffalse
Therefore,the Taylor expansion with respect to $e^{\mathrm{i}\theta}$, $\widetilde{f}$, $\widetilde{g}$, $\widetilde{p}$, and $\widetilde{q}$ is shown below,
\begin{subequations}
\begin{align*}
\mathfrak{T}(e^{\mathrm{i}\theta})&=1+\theta-\frac{\theta^2}{2}-\frac{\theta^3}{6}+O(\theta^4)
\end{align*}
\begin{align*}
\mathfrak{T}(\widetilde{f})&=\frac{1}{e}+\frac{\theta}{e}-\frac{3\theta^2}{8e}-\frac{\theta^3}{24e}+\mathcal{O}(\theta^4),\quad \mathfrak{T}(\widetilde{p})=e^{-\frac{2}{\theta}+\mathcal{O}(\theta)}\left(1+\theta-\frac{\theta^2}{2}-\frac{\theta^3}{6}+\mathcal{O}(\theta^4)\right)\\
\mathfrak{T}(\widetilde{g})&=1+\theta-\frac{3\theta^2}{4}-\frac{5\theta^3}{12}+O(\theta^4),\quad \mathfrak{T}(\widetilde{q})=e^{-\frac{8}{\theta^3}+\mathcal{O}(\theta)}\left(1+\theta-\frac{\theta^2}{2}-\frac{\theta^3}{6}+\mathcal{O}(\theta^4)\right)
\end{align*}
\end{subequations}
\fi
In Table~\ref{tab:Taylor}, only $\widetilde{g}$, or more accurately $g(\theta)$, has a first-order approximation to the angular function $\theta$ when the instantaneous rate of change of direction crosses the edges $[\mathbf{v}_i, \mathbf{v}_j]$, the others all have jump discontinuities. Thus, we can see that the associated mean curvature derived in \eqref{ansatz} will be more sensitive and accurate compared to other discretizations. 

The solution accuracy depends on the discretization error and the solution error. In terms of discretization error, a given mesh is a discrete approximation of the surface, so even if the equations are solved exactly, only an approximate solution is provided. By reducing the mesh size, the mesh will always be closer to the original surface and the geometry will vary more smoothly and reflect the surface structure. For PDEs, such as \eqref{second}, multiple iterations over the entire mesh are required. One way to prevent the computation from ending earlier before the equations are solved exactly is to improve the quality of the mesh, but at the same time it must be ensured that the approximate order of the geometrical quantities is not less than the approximation order of mesh. Thus, in terms of solution error, we provide a higher order approximation of the mean curvature $|\vec{H}|$ on the edges to reduce the solution error and thus improve the solution accuracy. Numerical applications are beyond the scope of this paper, and below we focus on how to obtain the associated mean curvature, i.e., a higher-order approximation of the mean curvature $|\vec{H}|$  on the edges.

As an important implication of the theory of minimal surfaces, there is another interpretation of the mean curvature vector. Let $f(x)$ for $x\in\mathcal{M}\subset\mathbb{R}^3$ be a regular parametrized surface and assume that $f$ is isothermal\footnote{
Due to the fact that the discrete Laplacian is defined by the energy functional \eqref{approach} and \eqref{dualL}, the discrete version of the isothermal parameterization can be interpreted infinitesimally (locally) on the edge $[f_i, f_j]$ as
\begin{align*}
|df|^2 dvol_{\mathcal{M}}\approx w_{ij}\lambda^2|f_i-f_j|^2.
\end{align*}
This means that the scaling transformations of the discrete Laplacian can imitate the isothermal parameterization. In the subsequent decomposition argument, different scaling transformations not only make the cotangent-weighted Laplacian (intrinsic term) invariant, but also make the mean curvature (extrinsic term) vary with the isothermal factor, which is consistent with the mean curvature theory of isothermal parameterization in \cite[Chap. 4]{osserman1986survey} 
}
. Let $\lambda^2=\langle f_{x_1}, f_{x_1}\rangle=\langle f_{x_2}, f_{x_2}\rangle$, $x=(x_1, x_2)\in\mathcal{M}$, then the equation \eqref{second} can be re-expressed as  
\begin{align*}
\Delta f=2\lambda^2\vec{H}.
\end{align*}
Assuming that $f$ is an isothermal regular parametrized surface, Theorem 3 in \cite{shefel1970totally} states that the surface $f(x)$ is a Riemann surface of class $C^k$, ($k\geq 2$ continuous derivatives) if and only if its mean curvature is a function of class $C^{k-2}$. In other words, the regularity between a surface and its mean curvature corresponds synchronously under an isothermal parameterization         . Based on the mean curvature theory of isothermal parameterization and Theorem 3 in \cite{shefel1970totally}, we can apply the idea of decomposition to improve the approximation order or so-called ``regularity" of the mean curvature.

The key step about deriving the associated mean curvature in ansatz \eqref{ansatz} is the inequality \eqref{ineq:20}. So we apply it to derive higher--order approximation.  Let $m>1$ and $n>1$ are real numbers such that $\frac{1}{m}+\frac{1}{n}=1$ and define
\begin{align*}
a:=\left|\left(\frac{\cot\alpha_{ij}+\cot\beta_{ij}}{2}\right)\mathbf{e}_{ij}\right|^{\frac{1}{m}}
,\quad
b:=\left|\frac{1-\cos\theta_{ij}^{kl}}{\sin\theta_{ij}^{kl}}\right|\left|\left(\frac{\cot\alpha_{ij}+\cot\beta_{ij}}{2}\right)\mathbf{e}_{ij}\right|^{\frac{n+1}{n}},
\end{align*}
By Young's inequality,
\iffalse
\begin{align*}
ab\leq
\frac{1}{m}\left|\left(\frac{\cot\alpha_{ij}+\cot\beta_{ij}}{2}\right)\mathbf{e}_{ij}\right|
+\frac{1}{n}\left|\frac{1-\cos\theta_{ij}^{kl}}{\sin\theta_{ij}^{kl}}\right|^n\left|\left(\frac{\cot\alpha_{ij}+\cot\beta_{ij}}{2}\right)\mathbf{e}_{ij}\right|^{n+1}.
\end{align*}
\fi
we give an ansatz like \eqref{ansatz} and the associated discrete mean curvature on the edges $[\mathbf{v}_i, \mathbf{v_j}]$ as follows,
\begin{align*}
\widetilde{w}_{ijkl}
=&\frac{|[\mathbf{v}_i, \mathbf{v}_j]|}{4n}\left[(n-1)\left(\frac{\cot\alpha_{ij}+\cot\beta_{ij}}{2}\right)^{\top}\right]\nonumber\\
&\quad\quad+\frac{|[\mathbf{v}_i, \mathbf{v}_j]|}{4n}\left(\frac{1-\cos\theta_{ij}^{kl}}{\sin\theta_{ij}^{kl}}\right)^n
\left[\left(\frac{\cot\alpha_{ij}+\cot\beta_{ij}}{2}\right)^{\perp}\right]^{n+1},\\
|\vec{H}_a|
:=&\frac{1}{n-1}\left(\frac{1-\cos\theta_{ij}^{kl}}{\sin\theta_{ij}^{kl}}\right)^n\left[\left(\frac{\cot\alpha_{ij}+\cot\beta_{ij}}{2}\right)^{\perp}\right]^{n+1}\\
=&\frac{1}{n-1}\left(\frac{1-\cos\theta_{ij}^{kl}}{\sin\theta_{ij}^{kl}}\right)^n\sin^{n+1}\left(\frac{\pi-\theta_{ij}^{kl}}{2}\right).
\end{align*}
In Table~\ref{tab:HiOrd}, we take $n=10, 50, 100$ as an example to illustrate the approximation order of the angular function and in Figure~\ref{HigherOder}, each red curve represents an integer $n$ from 2 to 100, and the blue line indicates pointwise convergence to a semicircle. 
\begin{table}[t]
\centering
\caption{The $(n-1)$th order approximation $|\vec{H}_a|$ for angular transients $\theta$}
\label{tab:HiOrd}
    \begin{tabular}{|c|l|l|}
        \hline
\thead{n} & \thead{Associated mean curvature $f:=|\vec{H}_a|$}& \thead{Difference $\mathfrak{T}(\widetilde{f})-\mathfrak{T}(e^{\mathrm{i}\theta})$} \\
        \hline
$10$ & $\frac{1}{9}\left(\frac{1-\cos\theta}{\sin\theta}\right)^{10}\sin^{11}\left(\frac{\pi-\theta}{2}\right)$ & $\mathfrak{T}(\widetilde{f})-\mathfrak{T}(e^{\mathrm{i}\theta})=-\frac{1}{9216}\theta^{10}+O(\theta^{11})$ \\
        \hline
$50$ & $\frac{1}{49}\left(\frac{1-\cos\theta}{\sin\theta}\right)^{50}\sin^{51}\left(\frac{\pi-\theta}{2}\right)$ & $\mathfrak{T}(\widetilde{f})-\mathfrak{T}(e^{\mathrm{i}\theta})=c_{50}\theta^{50}+\mathcal{O}(\theta^{51})$\\
        \hline
$100$ & $\frac{1}{99}\left(\frac{1-\cos\theta}{\sin\theta}\right)^{100}\sin^{101}\left(\frac{\pi-\theta}{2}\right)$ & $\mathfrak{T}(\widetilde{f})-\mathfrak{T}(e^{\mathrm{i}\theta})=c_{100}\theta^{100}+\mathcal{O}(\theta^{101})$\\
        \hline
    \end{tabular}
\end{table} 
\begin{figure}
   \includegraphics[width=1.\linewidth]{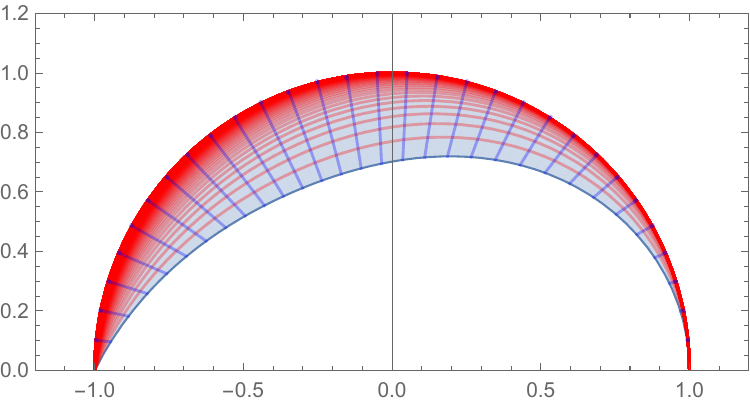}
   \caption{The $(n-1)$th order approximation $|\vec{H}_a|$ of the angular transient $\theta$ under the M\"{o}bius transform $e^{\mathrm{i}z}$  for $2\leq n\leq 100$.}
   \label{HigherOder}
\end{figure}

\section{Conclusions and future work}
First, this paper is based on the projective view of geometric metric theory, which means that the extrinsic Laplacian operator $\Delta_{\mathbb{R}^3}$ can be decomposed into an intrinsic Laplacian operator $\Delta_{\mathbb{R}^2}$ and an extrinsic mean curvature term $\vec{H}$. According to this theory, we study two discrete Laplace operators in $\mathbb{R}^3$, the primal Laplace operator and the dual Laplace operator. After providing a priori estimates, we define the associated Laplacian operator and then rigorously prove that the dual Laplace operator is the optimal approximation of the smooth extrinsic Laplacian--Beltrami operator in $\mathbb{R}^3$ compared to the primal Laplace operator.

Second, for the associated mean curvature, we analyze the composition of its expression and find out its essential geometric significance. Furthermore, we investigate the approximate order of the angular function and realize that, in contrast to other discrete mean curvatures, it not only has a first--order approach, but also provides a higher--order approximation based on geometrical facts about hypersurfaces \eqref{second}.

In the future, we will study conformal mean curvature flow, area--preserving mean curvature flow and conformal Wilmore flow using \eqref{second}. These geometric flow problems play an important role in the classification of surface states. Furthermore, the Euler-Lagrange equations for their corresponding energy functions suggest that the study of these geometric flow problems, which are closely related to the Laplace-Beltrami operator and the mean curvature vector, is an important starting point.

\backmatter

\begin{appendices}
\section{Apriori Estimates}\label{AppendixB}
Under the Delaurnay triangulation, we have the following estimates on the dual Laplacian of a tetrahedron $[\mathbf{v}_i, \mathbf{v}_j, \mathbf{v}_k, \mathbf{v}_l]$
\begin{itemize}
\item\begin{align*}
\cot\alpha_{ij}\cot\beta_{ij}\leq|\cot\alpha_{ij}||\cot\beta_{ij}|\leq\frac{\cot^2\alpha_{ij}+\cot^2\beta_{ij}}{2},
\end{align*}
where the equality holds if and only if $\cot\alpha_{ij}=\cot\beta_{ij}$.
\item We take $\mathrm{K}$ as the part of \eqref{dualL3} that relates to dihedral and face angles,
\begin{align*}
\mathrm{K}\equiv\left(\frac{2\cot\alpha_{ij}\cot\beta_{ij}}{\cos\theta_{ij}^{kl}}-(\cot^2\alpha_{ij}+\cot^2\beta_{ij})\right).
\end{align*}
For $\theta_{ij}^{kl}\in [0, \pi/2)$ and the unit $\mathbf{e}_{ij}:=\frac{[\mathbf{v}_i, \mathbf{v}_j]}{|[\mathbf{v}_i, \mathbf{v}_j]|}$, the upper bound for the operator $\mathrm{K}$ can be estimated as follows, 
\begin{align*}
|\mathrm{K}\mathbf{e}_{ij}|\leq\left(\frac{1}{\cos\theta_{ij}^{kl}}-1\right)M(\alpha_{ij}, \beta_{ij}, \mathbf{e}_{ij}).
\end{align*}
For $\theta_{ij}^{kl}\in (\pi/2, \pi]$ and the unit $\mathbf{e}_{ij}:=\frac{[\mathbf{v}_i, \mathbf{v}_j]}{|[\mathbf{v}_i, \mathbf{v}_j]|}$,
\begin{align*}
|\mathrm{K}\mathbf{e}_{ij}|\leq\left|\frac{1}{\cos\theta_{ij}^{kl}}-1\right|M(\alpha_{ij}, \beta_{ij}, \mathbf{e}_{ij}).
\end{align*}
where the nonnegative function $M(\alpha_{ij}, \beta_{ij}, \mathbf{e}_{ij})$ is determined by
\begin{align*}
M(\alpha_{ij}, \beta_{ij}, \mathbf{e}_{ij})
\equiv\max\left\{\left|(\cot^2\alpha_{ij}+\cot^2\beta_{ij})\mathbf{e}_{ij}\right|, \left|2\cot\alpha_{ij}\cot\beta_{ij}\mathbf{e}_{ij}\right|\right\}.
\end{align*}
Combining both cases, we choose the average upper bound estimate of $\mathrm{K}$
\begin{align}
\label{ineq:upper}
\left|\mathrm{K}\mathbf{e}_{ij}\right|\leq2\left|\frac{1}{\cos\theta_{ij}^{kl}}-1\right|\left|\left(\frac{\cot\alpha_{ij}+\cot\beta_{ij}}{2}\right)\mathbf{e}_{ij}\right|^2.
\end{align}
\begin{remark}
For completeness, we verify the validity of the upper bound \eqref{ineq:upper} under the Delaurnay condition.These a priori estimates have little impact on the main content of the paper and can be skipped if necessary. The derivation is in two parts.
\begin{claim*}
Given $\theta_{ij}^{kl}\in [\pi/2, \pi]$. If
$
A:=2\left(\frac{1}{\cos\theta_{ij}^{kl}}-1\right)\left(\frac{\cot\alpha_{ij}+\cot\beta_{ij}}{2}\right)^2
$, 
%=&\frac{1-\cos\theta_{ij}^{kl}}{\cos\theta_{ij}^{kl}}\frac{\cot^2\alpha_{ij}+\cot^2\beta_{ij}+2\cot\alpha_{ij}\cot\beta_{ij}}{2}\\
%=&\frac{1}{\cos\theta_{ij}^{kl}}\left[\left(\frac{1-\cos\theta_{ij}^{kl}}{2}\right)\left(\cot^2\alpha_{ij}+\cot^2\beta_{ij}+2\cot\alpha_{ij}\cot\beta_{ij}\right)\right]
$
B:=\frac{2\cot\alpha_{ij}\cot\beta_{ij}}{\cos\theta_{ij}^{kl}}-\cot^2\alpha_{ij}-\cot^2\beta_{ij}
%=&\frac{1}{\cos\theta_{ij}^{kl}}\left[2\cot\alpha_{ij}\cot\beta_{ij}-\cot^2\alpha_{ij}\cos\theta_{ij}^{kl}-\cot^2\beta_{ij}\cos\theta_{ij}^{kl}\right].
$
, then $A\leq B$.
\end{claim*}
\begin{proof}
We check directly
\begin{align*}
A-B
=&\frac{1-\cos\theta_{ij}^{kl}}{\cos\theta_{ij}^{kl}}\frac{\cot^2\alpha_{ij}+\cot^2\beta_{ij}+2\cot\alpha_{ij}\cot\beta_{ij}}{2}\\
&\quad\quad-\frac{1}{\cos\theta_{ij}^{kl}}\left(2\cot\alpha_{ij}\cot\beta_{ij}-\cot^2\alpha_{ij}\cos\theta_{ij}^{kl}-\cot^2\beta_{ij}\cos\theta_{ij}^{kl}\right)\\
\iffalse
=&\frac{1}{\cos\theta_{ij}^{kl}}\left[\left(\frac{1-\cos\theta_{ij}^{kl}}{2}\right)\left(\cot^2\alpha_{ij}+\cot^2\beta_{ij}+2\cot\alpha_{ij}\cot\beta_{ij}\right)\right]\\
&\quad\quad-\frac{1}{\cos\theta_{ij}^{kl}}\left[2\cot\alpha_{ij}\cot\beta_{ij}-\cot^2\alpha_{ij}\cos\theta_{ij}^{kl}-\cot^2\beta_{ij}\cos\theta_{ij}^{kl}\right]\\
\fi
=&\frac{1+\cos\theta_{ij}^{kl}}{2\cos\theta_{ij}^{kl}}\left[(\cot^2\alpha_{ij}+\cot^2\beta_{ij})-2\cot\alpha_{ij}\cot\beta_{ij}\right]
\end{align*}
\begin{enumerate}[label=\arabic*)]
\item For $\cot\alpha_{ij}\geq 0$ and $\cot\beta_{ij}\geq 0$, the bracketed term $[\cdot]$ is nonnegative so under the assumption that $\theta_{ij}^{kl}\in [\pi/2, \pi]$, say $\cos\theta_{ij}^{kl}\leq 0$, we obtain $A-B\leq 0$.
\item For $\cot\alpha_{ij}\geq 0$ and $\cot\beta_{ij}\leq 0$, the bracketed term $[\cdot]$ is also nonnegative so $A-B\leq 0$.
\end{enumerate}
\end{proof}
\begin{claim*} 
Under the Delaurnay condition, there exists $\theta_0\geq \pi/2$ so that $B\leq0$ for $\theta_{ij}^{kl}\in [\theta_0, \pi]$.
\end{claim*}
\begin{proof}
Proof by contradiction . We choose $\theta_{ij}^{kl}=\pi$, which contradicts the AM-GM inequality, so we can perturb $\pi$ slightly and then $B\leq0$ for $\theta_{ij}^{kl}\in [\theta_0, \pi]$.
\end{proof}
\end{remark}
\end{itemize}

Next, we consider a matrix (an operator) $\mathbf{L}:= (w_{ij})$ where $w_{ij}$ is determined by \eqref{dualL2} and \eqref{dualL3}. So $\mathbf{L}$ acts on the unit edge $\mathbf{e}_{ij}$ as follows,
\begin{align}
|w_{ijkl}\mathbf{e}_{ij}|
=&\frac{|[\mathbf{v}_i, \mathbf{v}_j]|}{8}\left|\cot\theta_{ij}^{kl}\left(\frac{2\cot\alpha_{ij}\cot\beta_{ij}}{\cos\theta_{ij}^{kl}}-(\cot^2\alpha_{ij}+\cot^2\beta_{ij})\right)\mathbf{e}_{ij}\right|\nonumber\\
\leq&\frac{|[\mathbf{v}_i, \mathbf{v}_j]|}{8}\left|\cot\theta_{ij}^{kl}\left(\frac{1}{\cos\theta_{ij}^{kl}}-1\right)\right|\left|\frac{(\cot\alpha_{ij}+\cot\beta_{ij})^2}{2}\mathbf{e}_{ij}\right|\label{ineq:19}\\
=&\frac{|[\mathbf{v}_i, \mathbf{v}_j]|}{4}\left|\frac{1-\cos\theta_{ij}^{kl}}{\sin\theta_{ij}^{kl}}\right|\left|\left(\frac{\cot\alpha_{ij}+\cot\beta_{ij}}{2}\right)^2\mathbf{e}_{ij}\right|\nonumber\\
%=&\textcolor{cyan}{\frac{|[\mathrm{v}_i, \mathrm{v}_j]|}{4}\left|\frac{1-\cos\theta_{ij}^{kl}}{\sin\theta_{ij}^{kl}}\right|\left|\frac{\cot\alpha_{ij}+\cot\beta_{ij}}{2}\right|^{\frac{3}{2}}\left|\frac{\cot\alpha_{ij}+\cot\beta_{ij}}{2}\right|^{\frac{1}{2}}}\nonumber\\
%\leq&\textcolor{cyan}{\frac{|[\mathrm{v}_i, \mathrm{v}_j]|}{4}\frac{\left[\left(\frac{1-\cos\theta_{ij}^{kl}}{\sin\theta_{ij}^{kl}}\right)\left(\frac{\cot\alpha_{ij}+\cot\beta_{ij}}{2}\right)^{3/2}\right]^2+\left[\left(\frac{\cot\alpha_{ij}+\cot\beta_{ij}}{2}\right)^{1/2}\right]^2}{2}}\nonumber\\
\leq&\frac{|[\mathbf{v}_i, \mathbf{v}_j]|}{8}\left|\frac{\cot\alpha_{ij}+\cot\beta_{ij}}{2}\mathbf{e}_{ij}\right|\nonumber\\
&+\frac{|[\mathbf{v}_i, \mathbf{v}_j]|}{8}\left|\frac{1-\cos\theta_{ij}^{kl}}{\sin\theta_{ij}^{kl}}\right|^2\left\|\mathbf{L}(\mathbf{e}_{ij})\right\|^2\left|\frac{\cot\alpha_{ij}+\cot\beta_{ij}}{2}\mathbf{e}_{ij}\right|\label{ineq:20},
%=&|\widetilde{w}_{ijkl}e_{ij}|\nonumber,
\end{align}
where the operator norm of $\mathbf{L}$ acting on the unit edge $\mathbf{e}_{ij}$ is defined as
\begin{align}\label{eq:opnorm}
    \|\mathbf{L}(\mathbf{e}_{ij})\|\equiv\sup_{|\mathbf{e}_{ij}|=1}\frac{|\mathbf{L}\mathbf{e}_{ij}|}{|\mathbf{e}_{ij}|}=\sup_{|\mathbf{e}_{ij}|=1}\left|\left(\frac{\cot\alpha_{ij}+\cot\beta_{ij}}{2}\right)\mathbf{e}_{ij}\right|.
\end{align}
Note that the last inequality in \eqref{ineq:20} can be derived by applying AM--GM inequality to $a$ and $b$ which are defined respectively as 
\begin{align*}
a:=\left|\left(\frac{\cot\alpha_{ij}+\cot\beta_{ij}}{2}\right)\mathbf{e}_{ij}\right|^{\frac{1}{2}}
,\quad
b:=\left|\frac{1-\cos\theta_{ij}^{kl}}{\sin\theta_{ij}^{kl}}\right|\left|\left(\frac{\cot\alpha_{ij}+\cot\beta_{ij}}{2}\right)\mathbf{e}_{ij}\right|^{\frac{3}{2}}.
\end{align*}
Finally, the first equality in \eqref{ineq:19} holds if and only if
$
\cot\alpha_{ij}=\cot\beta_{ij}
$
and the second equality in \eqref{ineq:20} holds if and only if
\begin{align*}
\left|\frac{1-\cos\theta_{ij}^{kl}}{\sin\theta_{ij}^{kl}}\right|
=&\frac{1}{\|\mathbf{L}(\mathbf{e}_{ij})\|}=\frac{1}{\sup_{|\mathbf{e}_{ij}|=1}\left|\left(\frac{\cot\alpha_{ij}+\cot\beta_{ij}}{2}\right)\mathbf{e}_{ij}\right|}, \quad\text{(by \eqref{eq:opnorm})}\\
\stackrel{\dag}{=}&\frac{1}{|\vec{H}_{\mathbf{e}_{ij}}|}=\frac{1}{\sin((\pi-\theta_{ij}^{kl})/2)}, \quad\text{(by \eqref{discreteE2})}.
\end{align*}
Up to this point, the above argument is just a series of algebraic derivations without introducing any geometric perspective. Since we want the second equality to hold true, the geometric and the embedding relationships between the submanifold, the tetrahedron $[\mathbf{v}_i, \mathbf{v}_j, \mathbf{v}_k, \mathbf{v}_l]$, and the ambient space $\mathbb{R}^3$ have to be treated carefully. Thus, we further consider Definition~\ref{discreteEdge} and Remark~\ref{rmkVE}, and then these observations provide us with clues to require the third equality $\dag$ and establish the ansatz \eqref{ansatz}.




\end{appendices}


%\bibliography{reference}% common bib file

\end{document}